%% file: arxiv.tex
\documentclass[preprint]{imsart}
\input{commandes_guillaume}

\begin{document}
	
	\begin{frontmatter}
		
		\title{Robust learning and complexity dependent bounds for regularized problems}
		\runtitle{Robust learning and complexity dependent bounds for regularized problems}
		
		\begin{aug}
			\author{\fnms{CHINOT}  \snm{Geoffrey}\thanksref{a,e1}\ead[label=e1,mark]{geoffrey.chinot@ensae.fr}}

			\runauthor{CHINOT Geoffrey}
			
			\affiliation{ENSAE, CREST, Institut Polytechnique de Paris}
			
			\address[a]{ENSAE, CREST, 5 avenue Henri Chatelier 91120 Palaiseau, France. \printead{e1}}
			
		\end{aug}
		
		\begin{abstract}
				We study Regularized Empirical Risk Minimizers (RERM) and minmax Median-Of-Means (MOM) estimators where the regularization function $\phi(\cdot)$ is an even convex function. We obtain bounds on the $L_2$-estimation error and the excess risk that depend on $\phi(f^*)$, where $f^*$ is the minimizer of the risk over a class $F$. The estimators are based on loss functions that are both Lipschitz and convex. Results for the RERM are derived under weak assumptions on the outputs and a sub-Gaussian assumption on the class $\{ (f-f^*)(X), f \in F  \}$. Similar results are shown for minmax MOM estimators in a close setting where outliers may corrupt the dataset and where the class $\{ (f-f^*)(X), f \in F  \}$ is only supposed to satisfy weak moment assumptions, relaxing the sub-Gaussian and the i.i.d hypothesis necessary for RERM. The analysis of RERM and minmax MOM estimators with Lipschitz and convex loss funtions is based on a weak local Bernstein Assumption. We obtain two “meta theorems" that we use to study linear estimators regularized by the Elastic Net. We also examine Support Vector Machines (SVM), where no sub-Gaussian assumption is required and when the target $Y$ can be heavy-tailed, improving the existing literature.
		\end{abstract}
		
		\begin{keyword}
			\kwd{Regularized learning, sharp oracle inequality, Kernel Method, Robustness, Median-of-means, elastic-net}
		\end{keyword}
		
	\end{frontmatter}

\section{Introduction}\label{sec:Intro}
On one hand, real world data analysis problems require nonlinear methods to model complex dependencies between random variables. On the other hand, linear models are well-understood and easy to implement, even in high dimension \cite{bishop2006pattern}. Over the last two decades, learning with positive definite kernels have become very popular in machine learning \cite{shawe2004kernel,scholkopf1999advances,steinwart2008support}. This popularity can be explained because kernel methods combine these advantages. Kernels can be used to model non linear dependencies, mapping them to a (usually high-dimensional) feature space. In this space, the estimation is linear. In this sense, kernel methods extend well-understood, linear statistical learning technics to real-world, complicated, structured, high-dimensional data based on a rigorous mathematical framework leading to practical modelling tools and algorithms. They have been used in many different fields such as finance \cite{chalup2008kernel}, biology \cite{scholkopf2004support,ben2005kernel,noble2004support}, econometry \cite{li2007nonparametric}, computer vision \cite{yang2000face}. \\
Let $(X,Y)$ be a random variable with distribution $P$ and $\cH_K$ a Reproducible Kernel Hilbert Space (RKHS) associated to a positive definite kernel $K$. Kernel methods consist in computing $f^*$ in $\cH_K$ such that the \textit{risk} $\cR(f) := \bE_{(X,Y) \sim P} [\ell(f(X),Y )]$ is minimized in $f^*$, where $\ell(f(X),Y )$ measures the error of predicting $ f(X)$ while the true answer is $Y$. However, the distribution $P$ is unknown and the minimization of the risk, necessary to compute $f^*$, is impossible in practice. To proceed, one is given a dataset $\cD = (X_i,Y_i)_{i=1}^N$ of random variables. Using the dataset $\cD$, kernel methods compute $\hat f_N^{\lambda}$ in $\cH_K$ such that
\begin{equation}  \label{def1_rkhs}
\hat f_N^{\lambda} \in \argmin_{f \in \cH_K} \frac{1}{N} \sum_{i=1}^{N}  \ell(f(X_i),Y_i) + \lambda \|f\|_{\cH_K}^2 \enspace,
\end{equation}
where $ \|f\|_{\cH_K}$ is the norm of $f$ in $\cH_K$ and $\lambda \geq 0$ is an hyperparameter to be tuned. The regularization term $ \lambda \|f\|_{\cH_K}^2$ controls the smoothness of $\hat f_N^{\lambda}$ through the value of $\lambda$. This regularization term is introduced to avoid “overfitting" since kernels provide enough flexibility to fit training data exaclty. The value of $\lambda$ balance the bias and the variance of $\hat f_N^{\lambda}$. Theoritical properties of kernel methods have been widely studied \cite{shawe2004kernel,scholkopf1999advances,steinwart2008support}.  Non-asymptotic bounds on the $L_2(\mu)$-\textit{error rate} $\|f^*-\hat f_N^{\lambda} \|_{L_2(\mu)}$, where $\mu$ denotes the marginal distribution of $X$, have been obtained for the quadratic loss function in \cite{mendelson2010regularization,smale2007learning,wu2006learning}. These bounds depend on the decay of eigenvalues of the kernel (at the population level) and are obtained for bounded continuous kernels but under the restrictive assumption that the random variable $Y \in [-M,M]$ almost surely. In~\cite{caponnetto2007optimal}, also for the quadratic loss function, the authors do not assume that $|Y|$ is bounded but that $Y-f^*(X)$ admits a Laplace transform. In this paper, we recover the same error rates as  \cite{mendelson2010regularization,caponnetto2007optimal} when the loss function $\ell$ is simultaneously Lipschitz and convex. We do not assume that $Y$ is bounded or $Y-f^*(X)$ is light-tailed. Our analysis uses a new localization technique developed in \cite{chinot} taking advantage of the convexity of the loss function $\ell$. Theorem~\ref{rkhs_informal} presents an informal result when $\ell$ is the absolute loss function. 
\begin{Theorem}[Informal] \label{rkhs_informal}
	Let $K$ be a bounded kernel. Assume that $Y = f^*(X) + W$ with $W$ a Cauchy random variable and $f^* \in \cH_K$, the RKHS associated with $K$. With probability larger than $1-\exp\big( -C_1N^{p/(p+1)} \big)$, for a well chosen value of $\lambda$ the estimator $\hat f $ associated to the absolute loss function defined in~\eqref{def1_rkhs} satisifies: 
	\begin{equation*}
		\|\hat f_N^{\lambda} - f^* \|_{L_2(\mu)}^2 \leq \frac{C_2}{N^{1/(1+p)}} \enspace,
	\end{equation*}
	where $C_1$ and $C_2$ are functions of the kernel and $\|f^*\|_{\cH_K}$. The value of $p \in (0,1)$ represents how fast the eigenvalues of the Kernel matrix decrease (see Section~\ref{sec_svm} for more precise arguments).
\end{Theorem}
Theorem~\ref{rkhs_informal} deals with a Cauchy noise but many different distributions can be handled with our analysis (see Theorem~\ref{thm_rkhs_quantile}). We obtain the same bounds as ~\cite{mendelson2010regularization,caponnetto2007optimal}. This is a first important contribution of this work. Fast rates for Kernel methods are derived even when the noise is heavy-tailed. Note also that nothing is assumed on the design $X$.\\

Kernel methods belong to the more general class of regularized methods, widespread in statistics and machine learning. 
These procedures date back to Tikhonov \cite{golub1979generalized}, and have been widely used in non-parametric statistics \cite{marsh2001spline,huang2003local} to smooth estimators. 
For example, the regularization $\phi(f) = \int (f'')^2 $ for spline estimators promotes smoothness by imposing regularity on the estimate. In kernel methods, the norm of a function in the RKHS controls how fast the function varies with respect to the geometry defined by the kernel. Consequently, the norm of regularization $\| \cdot \|_{\cH_K}$ is related with its degree of smoothness w.r.t. the metric defined by the kernel. Following the approach of~\cite{chinot}, we present an analysis for RERM with loss functions that are simultaneously Lipchitz and convex. The penalization function is not assumed to be a norm. It is simply required to be an even convex function. We derive bounds on the $L_2$-error and the excess loss for these general procedures. 
As far as we know, the only article considering a generic analysis of the RERM (with the quadratic loss) with a convex penalization is~\cite{lecue2017regularization}. However, their analysis does not hold for the square of a norm (see Assumption 5.1), which is a classical regularization methods in RKHS, see for instance~\cite{steinwart2008support}. By contrast, the new analysis presented in this paper covers many well-known methods such as kernel methods regularized by the square of a norm or the elastic net procedure \cite{zou2005regularization}.
The restriction here is that the loss function must be Lipschitz and convex. 
Both regression and classification problems can be addressed with our analysis. \\

Let $\cX, \cY$ be two measurable spaces such that $\cY \subset \bR$ and $(X,Y) \in \mathcal X \times \mathcal Y$ be random variables with joint distribution $P$. Let $\mu$ be the marginal distribution of $X$. For $E$ a linear subset of $L_2(X)$, let $F \subset E$ be a class of measurable functions $f: \cX \mapsto\bar{\cY}$ where $\bar{\cY} \subset \bR$ is convex (we do not have necessarily $\cY = \bar{\cY}$ for classification problems). In the standard learning framework, one would like to identify the best approximation to $Y$ using functions $f$ in the class $F$. To do so, let $\ell$ be a loss function, $\ell: F \times \mathcal X \times \mathcal Y \mapsto \bR$, $(f,x,y) \mapsto \ell_f(x,y) = \bar{\ell}(f(x),y)$ measuring the error made when predicting $y$ by $f(x)$, for $\bar{\ell}: \bar{\cY} \times \cY \mapsto \bR$. Let $f^{*} \in \argmin_{f \in F} R(f)$ where $R(f) := P\ell_f := \bE_P [ \ell_f(X,Y)] $. The \emph{oracle} $f^*$ provides the prediction of $Y$ with minimal risk among functions in $F$. Obviously, the distribution $P$ is unknown and minimizing the risk $R(f)$ over $f$ in $F$ is impossible in practice. Instead, one is given a dataset $\mathcal D = (X_i,Y_i)_{i=1}^N$ of random variables taking values in $\mathcal X \times \mathcal Y$. Using $\mathcal D$, the objective is to construct an estimator $\hat{f}_N$ such that the $L_2(\mu)$-\textbf{error rate}
\begin{equation*}
	\|\hat{f}_N-f^*\|_{L_2(\mu)} ^2= \bE \bigg[  \big(  \hat f_N (X) - f^*(X) \big)^2 | \cD  \bigg]
\end{equation*}
and the \textbf{excess risk}
\begin{equation*}
	P\cL_{\hat{f}_N} :=  (P\ell_{\hat f_N} - P\ell_{f^{*}}) | \cD = \bE_P \bigg[ \bar{\ell}(\hat f_N(X),Y) -  \bar{\ell}(f^*(X),Y) | \cD  \bigg] 
\end{equation*}
are small. While $P\cL_{\hat{f}_N}$ specifies the quality of prediction of the estimator $\hat{f}_N$, $	\|\hat{f}_N-f^*\|_{L_2(\mu)} $ quantifies the $L_2(\mu)$ approximation of the \emph{oracle} $f^*$ by the estimator $\hat f_N$. These two quantities being random, the results are derived with exponentially large probability. All along the paper, the following geometric Assumption is also granted. 
\begin{Assumption}\label{assum:convex}
	The class $F$ is convex.
\end{Assumption}
Assumption~\ref{assum:convex} imposes a geometric structure on the class $F$. This assumption is essential to use our “projection trick" and derive our main results. For example Assumption~\ref{assum:convex} holds when $F$ is a Hilbert space or the set of linear functionals in $\bR^p$, $F = \{ \inr{t,\cdot}: \; t \in \bR^p \}$. 
As in~\cite{chinot}, we consider Lipschitz and convex loss functions.
\begin{Assumption}\label{assum:lip_conv}
	There exists $L>0$ such that, for any $y \in \cY$, $\bar{\ell}(\cdot,y)$ is \textbf{$L$-Lipschitz} (see \eqref{Lip:cond}) and \textbf{convex} i.e for all $\alpha \in [0,1], (x,y) \in \cX \times \cY$ and $f,g \in F$, $\bar{\ell}(\alpha f(x) + (1-\alpha) g(x),y) \leq \alpha \bar{\ell}(f(x),y) + (1-\alpha )\bar{\ell}(g(x),y)$ 
\end{Assumption}
Assumption~\ref{assum:lip_conv} is satisfied in several examples, let us provide a short list of some of them.
\begin{itemize}
	\setlength{\itemsep}{1pt}
	\setlength{\parskip}{0pt}
	\setlength{\parsep}{0pt}
	\item The \textbf{logistic loss} defined, for any $u\in\bar{\cY}=\R$ and $y\in\cY=\{-1, 1\}$, by 
	$\ell(u,y) = \log(1+\exp(-yu))$ satisfies Assumption~\ref{assum:lip_conv} with $L=1$.
	\item The \textbf{hinge loss} defined, for any $u\in\bar{\cY}=\R$ and $y\in\cY=\{-1, 1\}$, by $\ell(u,y) = \max(1-uy,0)$ satisfies Assumption~\ref{assum:lip_conv} with $L=1$.
\end{itemize}
In those examples, the sets $\cY$ and $\bar{\cY}$ are different. The fact that every function $f$ in $F$ maps to the convex set $\bar{\cY}$ is crucial for the computation of the estimator $\hat{f}_N$ in practice~\cite{zhang2004statistical,bartlett2006convexity}
\begin{itemize}
	\item  The \textbf{Huber loss} defined, for any $\delta >0$, $u,y\in\cY=\bar{\cY}=\R$, by
	\[
	\ell(u,y) =  
	\begin{cases}
	\frac{1}{2}(y-u)^2&\text{ if }|u-y| \leq \delta\\
	\delta|y-u|-\frac{\delta^2}{2}&\text{ if }|u-y| > \delta
	\end{cases}\enspace,
	\]
	satisfies  Assumption~\ref{assum:lip_conv} with $L=\delta$.
	\item The \textbf{quantile loss} is defined, for any $\tau \in (0,1)$, $u,y\in\cY=\bar{\cY}=\bR$, by
	$\ell(u,y)  = \rho_{\tau}(u-y)$ where, for any $z\in \R$, $\rho_{\tau}(z) = z(\tau-I \{ z \leq 0\})$. It satisfies Assumption~\ref{assum:lip_conv} with $L=1$. For $\tau = 1/2$, the quantile loss is the $L_1$ loss.
	\item The \textbf{Hinge loss for regression} is defined for any $u,y\in\cY=\bar{\cY}=\bR$, by
	$\ell(u,y)  = \max(y-u,0)$. It satisfies Assumption~\ref{assum:lip_conv} with $L=1$. Note that the Hinge loss function is modified for regression problems. 
\end{itemize}
Classical results on the RERM in learning theory consider the quadratic loss function \cite{mendelson2014learning,lecue2017regularization,lecue2018regularization}. In this case $\bar \ell(u,v) = (u-v)^2/2$ for any $(u,v) \in \bar \cY \times \cY$. The starting point of their analysis is the following mutliplier/quadratic decomposition
\begin{equation*}
	\cL_f(X,Y) = (f(X)-Y)^2 - (f^*(X)-Y)^2 = (f(X)-f^*(X))^2 + 2 (f^*(X)-Y)(f(X)-f^*(X)) 
\end{equation*}
for any $f$ in $F$.
While the quadratic process $f \mapsto (f(X)-f^*(X))^2$ does not depend on the target $Y$, the multiplier process $f \mapsto (f^*(X)-Y)(f(X)-f^*(X))$ depends on the ``noise" $Y-f^*(X)$.
It can only be controlled under some restriction on this ``noise". For example, when $Y = g(X) +W$, where $g: \cX \mapsto \bR$ is a function in $F$ and $W$ is a random variable independent to $X$, we have $g = f^*$ and thus $Y-f^*(X) = W$. In this problem, bounding the multiplier process requires strong moment assumptions on the noise $W$ (see Theorem 1.2 in~\cite{mendelson2017multiplier}). If we replace the quadratic loss function by the absolute loss and if the noise is symmetric and independent to $X$ we also have $f^* = g$. In this case, from the Lipschitz property,
\begin{equation} \label{Lip:cond}
\forall (x,y) \in \cX \times \cY \mbox{   and   } f,g \in F, \quad    |\bar{\ell}(f(x),y) - \bar{\ell}(g(x),y) | \leq L |f(x)-g(x)|  \quad \mbox{ for   }L>0 \enspace,
\end{equation}
the multiplier process disappears. It becomes possible to handle heavy-tailed symmetric noise $W$. 
From~\eqref{Lip:cond}, note also that the random variable $Y$ does not need to be integrable. For instance, $W$ can be a Cauchy distribution. \\

To get fast rates of convergence, our analysis is based on the following local Bernstein condition
\begin{equation*}
	\forall f \in F: \|f-f^*\|_{L_2(\mu)} = r \mbox{   and   } \phi(f-f^*) \leq \rho, \quad AP\cL_f \geq  \|f-f^*\|_{L_2(\mu)}^2
\end{equation*}
where $r,\rho>0$. In the sequel, we have respectively $r$ and $\rho$ of the order of the error rate and $\phi(f^*)$, where we recall that $\phi(\cdot)$ is the regularization function and $f^*$ the oracle. This condition states that the excess risk $f \mapsto P\cL_f$ is $1/A$-strongly convex in a neighborhood of the oracle $f^*$. This new \textbf{local} Bernstein condition introduced in~\cite{chinot} is the cornerstone to obtain fast rates of convergence for settings where the noise may be heavy-tailed. Contrary to the analysis for the quadratic loss function, no Small Ball assumption is required~\cite{mendelson2014learning,lecue2017regularization}. In addition to handle heavy-tailed noise, the use of Lipschitz function significantly simplifies the proof since only one process has to be considered. The main argument of the proof is a new “projection trick" (see the sketch of proof in Section~\ref{sec_erm}) making the proof simpler. For example, no peeling technic is required. To summarize, the \textbf{contributions} of our new analysis for the RERM are the following
\begin{itemize}
	\setlength{\itemsep}{1pt}
	\setlength{\parskip}{0pt}
	\setlength{\parsep}{0pt}
	\item We consider very general convex regularization functions $\phi(\cdot)$.
	\item For Lipschitz and convex loss function, heavy-tailed noise can be handled.
	\item Our proof relies on a convex argument simple to understand.
\end{itemize}
The RERM are robust with repsect to the noise of the problem as long as the loss function is Lipschitz. However a single outlier in the $X_i$ may make the RERM really bad. In addition, the RERM performs well only when the empirical excess of risk $f \mapsto P_N \cL_f$ uniformly concentrates around its expectation $f \mapsto P\cL_f$. To do so, it is necessary to impose a strong concentration assumption on the class $ \{  \cL_f(X,Y), f \in F\}$. From Assumption~\ref{assum:lip_conv} it is implied by a concentration assumption on the class $\{(f - f^*)(X), f \in F \}$. Consequently, sub-Gaussian or boundedness assumptions are necessary on the class $\{(f-f^*)(X), f \in F  \}$ to obtain an exponentially large confidence for RERM.\\

RERM serves as benchmark for more advanced estimators. In a second time, we study regularized minmax MOM-estimators introduced in~\cite{lecue2017robust} for least-squares regression as an alternative to other MOM-based procedures \cite{LugosiMendelson2016, LugosiMendelson2017, lugosi2019sub,lecue2017learning}. In the case of convex and Lipschitz loss functions, these estimators satisfy the following properties 1) as the RERM, they are efficient under weak assumptions on the noise 2) they achieve optimal rates of convergence under weak stochastic assumptions on the class $ \{  \cL_f(X,Y), f \in F\}$ and 3) the rates are not downgraded by the presence of some outliers in the dataset. These results are not surprising since it has already been observed in \cite{lecue2017robust,chinot}. Although attractive, mimmax MOM-estimators present some drawbacks. Their construction depends on the confidence level (through $K$). Under stronger moment assumptions,~\cite{minsker2018uniform} proposed a construction of MOM-based estimators independent to the confidence level. The  implementation of MOM-based estimators is still an open question even if good empirical results have been obtained in~\cite{lecue2017robust,lecue2018robust,chinot}.\\

The main theorems (for the RERM and the minimax MOM estimators) are general and can be applied for different applications. In particular, we study 1) the Elastic net regularization for linear estimators in $\bR^p$ and 2) kernel methods in RKHS associated to a bounded kernel. In particular, we extend the results from~\cite{mendelson2010regularization,smale2007learning,wu2006learning,caponnetto2007optimal} for heavy-tailed noise.\\

To summarize, the \textbf{contributions of this paper} are the following:
\begin{itemize}
	\setlength{\itemsep}{1pt}
	\setlength{\parskip}{0pt}
	\setlength{\parsep}{0pt}
	\item We obtain an analysis for the RERM for general convex regularization functions under weak assumptions on the noise. This analysis is based on a local Bernstein assumption and holds under a strong concentration assumption on the class $\{(f-f^*)(X), f \in F  \}$.
	\item Under the same local Bernstein assumption, we study minimax MOM estimators and show that 1) as the RERM, they are efficient under weak assumptions on the noise 2) they achieve optimal rates of convergence under weak stochastic assumptions on the class $\{(f-f^*)(X), f \in F  \}$ and 3) the rates are not downgraded by the presence of some outliers in the dataset
	\item We apply this analysis to linear estimators regularized with elasitc net.
	\item Under the same local Bernstein assumption, with a slighlty different concentration argument, we study regularized learning problems in RKHS. The noise can be heavy-tailed and no sub-Gaussian on $\{(f-f^*)(X), f \in F  \}$ is required to get fast rates of convrgence. 
\end{itemize}

The paper is organized as follow. In Section~\ref{sec_erm} and~\ref{sec_mom} we respectively present general results for RERM and minmax MOM estimators. Section~\ref{sec_applications} is devoted to the application of our main theorems for the problems of linear estimators regularized with elastic net and Support vector machines. Section~\ref{sec_proof}-~\ref{app_supp} gather the proofs of the main theorems.

\paragraph{Notations:}
In the remaining of the paper, the following notations will be used repeatedly. We will write $L_2$ instead of $L_2(\mu)$, let $r>0$,
\[
rB_{L_2} = \{ f \in F: \|f(X)\|_{L_2(\mu)} \leqslant r \},\quad rS_{L_2} = \{ f \in F: \|f(X)\|_{L_2(\mu)} = r \}\enspace. 
\]
For any set $H$ for which it makes sense, let $H + f^{*} = \{ h+f^{*} \mbox{ s.t } h \in H   \}$, $H - f^{*} = \{ h-f^{*} \mbox{ s.t } h \in H   \}$. The notations $a \vee b$ and $a \wedge b$, will denote respectively $\max(a,b)$ and $\min(a,b)$.

\section{Regularized Empirical Risk Minimization (RERM) } \label{sec_erm}
All along this section, data $(X_i,Y_i)_{i=1}^N$ are \textbf{independent and identically distributed} with common distribution $P$. The unknown risks are estimated by their empirical counterparts, and the oracle is estimated by the \emph{empirical risk minimizer} (ERM) (see~\cite{MR2829871}), defined by 
\begin{equation*} 
	\hat{f}^{ERM} = \argmin_{f \in F} P_N \ell_f := \frac{1}{N} \sum_{i=1}^N \bar \ell(f(X_i),Y_i) \enspace.
\end{equation*}
Clearly, if the class $F$ is too small, there is no hope that $f^*(X)$ is close to $Y$. One has to consider large classes leading to large error rates. To bypass the fact that $F$ may be very large, we can use the classical approach of \emph{regularization} where the penalization function emphasizes the belief we may have on the \emph{oracle} $f^*$. It leads to the Regularized Empirical Risk Minimizer (RERM) defined as
\begin{equation} \label{def_erm}
\hat{f}^{RERM}_{\lambda} = \argmin_{f \in F} P_N \ell_f  + \lambda  \|f\| \enspace,
\end{equation}
where $\|\cdot\|: E \mapsto \bR^+$ is a norm. 
However, the estimators $\hat{f}^{RERM}_{\lambda} $ defined in~\eqref{def_erm} are rather restrictive since it does not cover penalizations which are not a norm such as $\|f\|_{\cH_K}^2$ (i.e the square of the norm in a reproducible Kernel Hilbert space) or the Elastic net procedure (see~\cite{zou2005regularization}). To bypass this limitation, the estimator defined in Equation~\eqref{def_erm} will be replaced by
\begin{equation} \label{def_erm2}
\hat{f}^{\phi}_{\lambda} = \argmin_{f \in F} P_N \ell_f  + \lambda  \phi(f) :=  \argmin_{f \in F} P_N \cL_f^{\lambda}
\end{equation}
where $\phi : E \mapsto \bR^+$ is a function satisfying the following Assumption.
\begin{Assumption} \label{assum:phi} Let $\phi: E \mapsto \bR^+$ be a real function such that
	\begin{itemize}
		\setlength{\itemsep}{1pt}
		\setlength{\parskip}{0pt}
		\setlength{\parsep}{0pt}
		\item $\phi$ is even,  convex and $\phi(0) = 0$
		\item 	There exists a constant $\eta >0$ such that for all $f,g \in F$ 
		\begin{equation}
		\phi(f+g) \leq \eta \big(\phi(f) + \phi(g) \big)
		\end{equation}
	\end{itemize}
\end{Assumption}
Assumption~\ref{assum:phi} holds for any norm but also for the square of a norm (with $\eta =2)$, the elasitc net penalization (with $\eta = 2$) defined for any $t$ in $\bR^p$ as $\phi(t) = (1-\alpha)\|t\|_1 + \alpha \|t\|_2^2$, where $\alpha \in [0,1]$, $\|t\|_1 = \sum_{i=1}^p |t_i|$ and $\|t\|_2^2 = \sum_{i=1}^p t_i^2$. To control the $L_2$-error rates for the RERM, it is necessary to impose a \textbf{concentration assumption} on the class $ \{  \cL_f(X,Y), f \in F\}$. From Assumption~\ref{assum:lip_conv} it is implied by a concentration assumption on the class $\{(f - f^*)(X), f \in F \}$ (this assumption will be relaxed using MOM-type estimators in Section~\ref{sec_mom}).
\begin{Definition}
	A class $F$ is called $B$ sub-Gaussian (with respect to $X$) for some constant $B \geq 0$ when for all $f$ in $F$ and for all $\lambda >1$
	\begin{equation*}
		\mathbb E \exp( \lambda |f(X)|/ \|f\|_{L_2} ) \leq \exp(\lambda^2 B^2/2)\enspace.
	\end{equation*}
\end{Definition}
\begin{Assumption}  \label{sub_ass}
	\label{ass:sub-gauss} The class $F-f^{*}$ is $B$ sub-Gaussian.
\end{Assumption}
For example, when $F$ is the class of linear functionals in $\bR^p$, $F = \{  \inr{\cdot,t}, \; t \in T \}$ for $T \subset \bR^p$, $F-f^*$ is $1$ sub-Gaussian if $X \sim \cN(0,\Sigma)$ or if $X = (x_j)_{j=1}^p$ has independent coordinates that are $1$ sub-Gaussian. In the sub-Gaussian framework, a natural way to measure the \emph{statistical complexity} of the function class $F$ is via the Gaussian mean-width that we introduce now.
\begin{Definition}\label{def:gauss_mean_width}
	Let $H\subset L_2$ and $(G_h)_{h\in H}$ be the canonical centered Gaussian process indexed by $H$, with covariance structure
	\[
	\forall h_1,h_2\in H, \qquad \left(\E (G_{h_1}- G_{h_2})^2\right)^{1/2} = \left(\E(h_1(X)-h_2(X))^2\right)^{1/2}\enspace. 
	\]
	The \textbf{Gaussian mean-width} of $H$ is $w(H) = \E \sup_{h\in H} G_h$.
\end{Definition}
For example, when $	F = \{  \inr{\cdot,t}, \; t \in \bR^p \}$, and $X \sim \cN(0,\Sigma)$, $w(T) = \bE \sup_{t \in T} \inr{G,t}$, where $T$ is a subset of $\bR^p$ and $G \sim  \cN(0,\Sigma)$. The Gaussian mean-width is closely related with metric complexities such as the entropy through the Sudakov's inequality, see Chapter 1 in~\cite{chafai2012interactions} for precise inequalities.\\
Following ideas developed in~\cite{lecue2017learning,lecue2017regularization,lecue2018regularization,mendelson2014learning}, the complexity parameter driving the statistical behavior of the estimator $\hat{f}^{\phi}_{\lambda}$ is defined as a fixed point depending on the Gaussian mean-width:
\begin{Definition}\label{def:function_r}
	The complexity is measured via a non-decreasing function $r(\cdot)$ such that for every $A >0$, 
	\begin{equation*}
		r(A) = \inf \bigg\{ r> 0: \;  32LB w \big(F \cap B_{\eta(4+2A^{-1}) \phi(f^*)}^{\phi}(f^*)  \cap (f^*+r B_{L_2})  \big) \leq   (2A)^{-1}\sqrt{N}r^2   \bigg\} 
	\end{equation*}
	where $B_{\delta}^{\phi}(g) =  \{f \in F : \; \phi(f-g) \leq \delta \}$ , $L$ is the Lipschitz constant of Assumption~\ref{assum:lip_conv}, $B$ is the sub-Gaussian constant defined in Assumption~\ref{ass:sub-gauss} and $\eta$ is defined in Assumption~\ref{assum:phi}.
\end{Definition}
Note that when $\phi$ is a norm, $B_{\delta}^{\phi}(g)$ simply corresponds to the ball of regularization centered in $g$ with radius $\delta$. We are now in position to introduce the \textbf{local Bernstein condition} allowing to derive fast rates of convergence for heavy-tailed problem. 
\begin{Assumption}\label{assum:fast_rates} There exists a constant $A^* > 0$ such that for all $f\in F$ if $\norm{f-f^*}_{L_2} = r(A^*)$ and $\phi(f-f^*) \leq \eta(4+2(A^*)^{-1}) \phi(f^*)$ then $\|f-f^{*}\|_{L_2}^2\leqslant A^*P\mathcal{L}_f $. 
\end{Assumption}
In the sequel of this section we will write $r^*$ instead of $r(A^*)$. 
Condition~\ref{assum:fast_rates} states that $f \mapsto P\cL_f$ is $1/A^*$-strongly convex in a subset of the $L_2$-sphere centered in $f^*$ with radius $r^*$. 
As explained in \cite{chinot}, this local Bernstein condition holds in examples where $F$ is not bounded in $L_2$-norm, and therefore, where the global Bernstein condition of \cite{pierre2019estimation}( $\|f-f^{*}\|_{L_2}^2\leqslant A^*P\mathcal{L}_f$ for all $f\in F$)  does not hold. Assumption~\ref{assum:fast_rates} replaces the small-ball Assumption (see \cite{mendelson2014learning} for instance) for learning problems with Lipschitz and convex loss functions.
In~\cite{chinot}, the authors consider non-regularized problems where the local Bernstein condition is required over the whole $L_2$-sphere of radius $r^*$. For regularized-procedure, this condition is required only for functions $f$ in this $L_2$-sphere of radius $r^*$ such that $\phi(f-f^*) \leq \eta(4+2(A^*)^{-1}) \phi(f^*)$. For instance, in the case of RKHS associated to a bounded kernel $K$, the condition $\phi(f-f^*) \leq \rho$, for $\rho>0$ implies that the function $f-f^*$ are bounded by $\sqrt{\rho \|K\|_{\infty}}$ (see Section~\ref{sec_svm}). This localization with respect to the regularization norm is essential to verify the local Bernstein Assumption in practice and obtain fast rates of convergence (see Section~\ref{sec_svm}).    \\
We are now in position to present the main theorems of this section. 

\begin{Theorem}\label{thm_erm_conv2}
	Grant Assumptions~\ref{assum:lip_conv},~\ref{assum:convex},~\ref{assum:phi},~\ref{ass:sub-gauss} and ~\ref{assum:fast_rates}. With probability larger than
	\begin{equation}\label{eq:proba}
	1-2\exp\bigg(-  \frac{N (r^*)^2}{4(32A^*LB)^2}  \bigg)
	\end{equation}
	for all regularization parameters $\lambda \geq \lambda_0 = (r^*)^2 / \phi(f^*) $ the estimator $\hat{f}^{\phi}_{\lambda}$ defined in Equation~\eqref{def_erm2} satisfies
	\begin{align*}
		\|\hat{f}^{\phi}_{\lambda} - f^{*}\|_{L_2} \leq  (4+6A^*)\lambda \frac{\phi(f^*)}{r^*}\\  \mbox{and} \quad  \phi(\hat{f}^{\phi}_{\lambda} - f^{*}) \leq  (4+2/A^*)\eta   \phi(f^*).  
	\end{align*}
\end{Theorem}
\begin{Remark}
	Theorem~\ref{thm_erm_conv2} holds for an exponentially large probability~\eqref{eq:proba} simultaneously for all $\lambda \geq \lambda_0$. As a consequence it can be used with a random choice of regularization parameter $\hat{\lambda}$ as long as $\{  \hat{\lambda} \geq \lambda_0 \}$ hold with large probability. For example, we could use a cross validation scheme to generate $\hat{\lambda}$.
\end{Remark}

Note that for $\lambda = \lambda_0$, we obtain $	\|\hat{f}^{\phi}_{\lambda} - f^{*}\|_{L_2} \leq  (4+6A^*)r^*$, which is the minimax rate into the class $\{f \in F: \phi(f) \leq \phi(f^*) \}$ (see \cite{lecue2017regularization}). Since we do not have access to $\phi(f^*)$, taking $\lambda_0$ is impossible. To bypass this issue we use a Lepski's adaptation method (see \cite{lepskii1992asymptotically, lepskii1993asymptotically,birge2001alternative}). To do so, the following assumption is required.
\begin{Assumption} \label{assum_bounded}
	There exists $M > 0$ such that $\phi(f^*) \leq M$.
\end{Assumption}	
Assumption~\ref{assum_bounded} is natural since regularization procedures are used when one believes that $\phi(f^*)$ is small. Since Theorem~\ref{thm_erm_conv2} holds with the same probability for all  $\lambda \geq \lambda_0$, one can choose $M$ very large in the Lepski's method without deteriorating the probability of the event. \\
For $j= 1, \cdots, J = M + \lceil \log_2(M) \rceil$, let us define $\phi_j = 2^j/2^M$, $\phi_0=0$ and $\lambda_j = r_j^2/\phi_j$ where 
\begin{equation*}
	r_j = \inf \big\{ r> 0: \;  32LB w \big(F \cap B_{\eta(4+2(A^*)^{-1}) \phi_j}^{\phi}(f^*)  \cap (f^*+r B_{L_2}) \big) \leq (2A^*)^{-1} \sqrt{N}r^2   \big\} 
\end{equation*}
Moreover for all $\lambda >0$ let us define 
\begin{gather*}
	T_{\lambda}(f) = P_N(\ell_f - \ell_{\hat{f}^{\phi}_{\lambda}}) + \lambda \big( \phi(f)- \phi(\hat{f}^{\phi}_{\lambda})  \big), \quad \hat{R}_j = \{ f \in F: \; T_{\lambda_j}(f) \leq \big( (A^*)^{-1} + 2 \big) \lambda_j \phi_j    \} 
	\\  k^* = \inf \{ k \in \big\{ 1, \cdots,J \}: \; \cap_{j \geq k}^{J} \hat{R}_j \neq \emptyset \big\} \quad \mbox{and set} \quad \tilde{f} \in \cap_{j \geq k^*}^{J} \hat{R}_j \enspace.
\end{gather*}
Using the Lepski's method we are in position to state to following theorem.
\begin{Theorem} \label{theo_lepski}
	Assumptions~\ref{assum:lip_conv},~\ref{assum:convex},~\ref{assum:phi},~\ref{ass:sub-gauss},~\ref{assum:fast_rates} and~\ref{assum_bounded},  with probability larger than 
	\begin{equation*}
		1-2\exp\big(-  \frac{N (r^*)^2}{4(64A^*LB(8+12A^*))^2} \big)
	\end{equation*}
	\begin{align*}
		& \| \tilde f - f^{*}\|_{L_2} \leq (8+12A^*) r^*, \quad  \phi(\tilde f - f^{*}) \leq   (4+2/A^*)\eta  \phi(f^*)\\
		&  \mbox{and} \quad P\cL_{\tilde{f}} \leq (4+3/A^*)(r^*)^2 \enspace.  
	\end{align*}
\end{Theorem}
Note that such a procedure required the knowledge of $A^*$ and $M$. 
Complete proofs of Theorem~\ref{theo_lepski} and Theorem~\ref{thm_erm_conv2} are presentend in Section~\ref{sec_proof} in the Appendix. Here we present a simple sketch of the proof of Theorem~\ref{thm_erm_conv2}. Our proof relies on a homogeneity argument allowing to study the empirical excess risk only in neighborhood around the oracle $f^*$. $\vspace{0.2cm}$	\newline
\textbf{Sketch of the proof : }  
The main arguments are presented up to some constants depending on $A^*,L$ and $\eta$. The proof is splitted into two parts. First, we identify a random event onto which the statistical behavior of $\hat{f}^{\phi}_{\lambda}$ can be studied using deterministic arguments. Next, we prove that this event holds with large probability. Here we will only focus on the deterministic argument (see Section~\ref{sec_proof} for the stochastic control).\\
Let  $\cB_{\lambda} =  \{f \in F: \; \|f-f^*\|_{L_2} \leq \lambda \phi(f^*)/r^* \mbox{ and }  \phi(f-f^*) \leq \phi(f^*) \}$ and the stochastic event is defined as
\begin{equation*}
	\Omega := \left\{ 
	\mbox{for all } f\in F\cap (f^*+ r^*B_{L_2}) \cap B_{\phi(f^*)}^{\phi}(f^*) , \quad \big|(P-P_N)\cL_f\big|\leq (r^*)^2  \right\}
\end{equation*} 
By definition, the estimator $\hat{f}^{\phi}_{\lambda}$ satisfies $P_N\cL_{\hat{f}^{\phi}_{\lambda}}^{\lambda} \leq 0$.
Therefore, to prove Theorem~\ref{thm_erm_conv2} it is sufficient to show that on $\Omega$, $	P_N\cL_{f}^{\lambda} > 0$ for all functions $f$ in $F \backslash  \cB_{\lambda}$. The proof follows from an homogeneity argument saying that for all functions $f \in  F \backslash \cB_{\lambda}$, there exist $f_0$ in the frontier of $ \mathcal B_{\lambda}$  and $\alpha \geq 1$ such that $P_N\cL^{\lambda}_{f} \geq \alpha P_N\cL^{\lambda}_{f_0} $. On the frontier of $\cB_{\lambda}$, either we have 1) $\phi(f_0-f^*)=  \phi(f^*)$ and $\norm{f_0-f^*}_{L_2}\leq \lambda \phi(f^*)/r^* $ or  2) $\norm{f_0-f^*}_{L_2}= \lambda \phi(f^*)/r^* $ and $\phi(f_0-f^*) \leq \phi(f^*)$. \\
The homogeneity argument linking the empirical excess risk of $f$ to the one of $f_0$ is the following. For all $i \in \{1,\cdots,N \}$, let $\psi_i: \mathbb R \rightarrow \mathbb R $ be defined for all $u\in \R$ by 
\begin{equation}\label{eq:fct_psi}
\psi_i(u) = \bar{\ell} (u + f^{*}(X_i), Y_i) - \bar{\ell} (f^{*}(X_i), Y_i).
\end{equation}
The functions $\psi_i$ are such that $\psi_i(0) = 0$, they are convex because $\bar{\ell}$ is, in particular $\alpha\psi_i(u) \leq \psi_i(\alpha u)$ for all $u\in\mathbb R$ and $\alpha \geq 1$ and $\psi_i(f(X_i) - f^{*}(X_i) )=  \bar{\ell} (f(X_i), Y_i) - \bar{\ell} (f^{*}(X_i), Y_i) $ so that the following holds:
\begin{align} \label{conv_arg}
	\nonumber P_N \cL_f & = \frac{1}{N} \sum_{i=1}^{N} \psi_i \big( f(X_i)- f^{*}(X_i) \big) = \frac{1}{N} \sum_{i=1}^{N}   \psi_i(\alpha( f_0(X_i)- f^{*}(X_i) ))\\
	&\geq \frac{\alpha}{N} \sum_{i=1}^{N}   \psi_i(( f_0(X_i)- f^{*}(X_i))) = \alpha P_N \cL_{f_0}.
\end{align}
For the regularization part, since $\alpha \geq 1$, the same homogeneity arguments holds. 	
\begin{align*}
	\phi(f) - \phi(f^{*})  = \phi \big(f^* + \alpha(f_0 - f^*) \big) - \phi(f^{*}) \geq \alpha \big( \phi(f_0) - \phi(f^*)  \big)
\end{align*}
It remains to control $P_N \cL_{f_0}^{\lambda}$ in the two cases 1) and 2). Up to technicalities, in case 1), we use Assumption \ref{assum:phi} to showing that $\phi(f_0)  - \phi(f^*) \geq \phi(f^*)$ (up to constants). Using the event $\Omega$, we show that $P_N \cL_{f_0} \geq  - \theta \lambda \phi(f^*) $ for $\theta > 0$ small enough. In case 2), we use that $\phi(f_0)  - \phi(f^*)  \geq - \phi(f^*)$ and the local Bernstein Assumption~\ref{assum:fast_rates} to prove that $P_N \cL_{f_0} \geq   \gamma \lambda \phi(f^*)$ for $\gamma > 0$ large enough which concludes the deterministic argument. $ \quad \blacksquare$

\section{Robustness to outliers and heavy-tailed data via Minmax MOM estimators} \label{sec_mom}

In Section~\ref{sec_erm}, we assumed that the class $\{(f - f^*)(X), f \in F \}$. is sub-Gaussian and that the data $(X_i,Y_i)_{i=1}^N$ are i.i.d with the same distribution $P$. In this section, we relax these assumptions using \textbf{minmax-MOM type estimators}. For any $i \in \{1,\cdots,N \}$, let $P_i$ be the distribution of $(X_i,Y_i)$. 
Let $\mathcal{I} \cup \mathcal{O}$ denote an unknown partition of $\{1,\cdots N\}$. The cardinality of $\mathcal{O}$ is denoted $|\mathcal{O}|$. Data $(X_i,Y_i)_{i \in \mathcal{O}}$ are considered as outliers. \textbf{No assumption} on the distribution $P_i$ for $i \in \mathcal{O}$ is made and can be dependent or even adversarial. The informative random variables $(X_i,Y_i)_{i \in \mathcal{I}}$ satisfy:
\begin{Assumption}\label{assum:moments}
	The data $(X_i,Y_i)_{i \in \cI}$ are independent and for all $ i  \in \mathcal{I}: P_i(f-f^{*})^2(X_i) = P(f-f^{*})^2(X) $ and $ P_i\mathcal{L}_f  = P\mathcal{L}_f $ where we recall that $P$ is the distribution of $(X,Y)$ \enspace.
\end{Assumption} 
Assumption~\ref{assum:moments} holds in the i.i.d framework but it covers other situations where informative data $(X_i,Y_i)_{i \in \cI}$ may not have the same distribution. It is only required to induce the same $L_2$-structure on the class $F$ and the same excess risk.\\

Let $(B_s)_{s=1,\ldots,S}$ denote a partition of $\{1,\ldots,N\}$ into blocks $B_s$ of equal size $N/S$ (if $N$ is not a multiple of $S$, just remove some data). Following~\cite{lecue2017robust} the minmax MOM-estimators are defined as 
\begin{equation}\label{def:MOM}
\hat f_S^{\lambda} = \argmin_{f \in F} \sup_{g \in F} MOM_S(\ell_f-\ell_g) + \lambda\big(  \phi(f)-\phi(g) \big),
\end{equation}
where $MOM_S(\ell_f - \ell_g) = \text{Med} \big(P_{B_1} (\ell_f - \ell_g),\cdots,P_{B_S} (\ell_f-\ell_g)\big)$ with $P_{B_s}(\ell_f-\ell_g) =(1/|B_s|) \sum_{i \in B_s} \ell_f(X_i,Y_i)-\ell_g(X_i,Y_i)$.\\
Since we no longer consider the sub-Gaussian framework, we have to adapt the complexity parameter to this new setup. The complexity is measured via a function $\tilde{r}(\cdot)$ defined as 
\begin{align} \label{comp:rad}
	\nonumber
	\tilde{r}(A) = \inf \bigg\{ & r >0 : \forall J \subset \mathcal{I} : |J| \geqslant N/2, \\ 
	&  \; \mathbb{E} {\sup_{f \in F \cap (f^*+rB_{L_2}) \cap B_{\eta(4+2A^{-1}) \phi(f^*)}^{\phi}(f^*)} \bigg |{\sum_{i \in J} \sigma_i (f-f^*)(X_i)} \bigg |} \leq (384AL)^{-1} r^2 |J| \bigg \} 
\end{align}
where $(\sigma_i)_{i=1}^N$ are i.i.d Rademacher random variables independent from $(X_i,Y_i)_{i \in \cI}$.\\
This complexity function is very close to the one in the sub-Gaussian case from Section~\ref{sec_erm} expect that the Rademacher-complexity replaces the Gaussian mean-width. When the class $F-f^*$ is $B$-sub-Gaussian, a standard chaining argument~\cite{talagrand2006generic} shows that $\tilde r (\cdot)$ and $r(\cdot)$ are equivalent. However, when only $L_p$ conditions are granted on the class $F-f^*$, $\tilde r(\cdot)$ may be larger than $r(\cdot)$, see~\cite{chinot}, for instance. It is also necessary to adapt the local Bernstein condition from Assumption~\ref{assum:fast_rates} to the MOM-framework
\begin{Assumption}\label{assum:fast_rates_MOM} 
	There exists a constant $\tilde{A} > 0$ such that, for all $f$ in $F$ satisfying $\|f-f^{*}\|_{L_2} =\sqrt{C_{S,r}(\tilde{A})}$ and $\phi(f-f^*) \leq \eta(4+2/\tilde A) \phi(f^*)$, then $\|f-f^{*}\|_{L_2}^2 \leqslant \tilde{A} P\mathcal{L}_f$ where 
	\begin{equation}
	\label{def_CK}	C_{S,r}(A) =  \max\left( \tilde{r}^2(A), 368A^2L^2\frac{S}{N}\right) \enspace.\\
	\end{equation}
\end{Assumption} 
As Assumption~\ref{assum:fast_rates}, Assumption~\ref{assum:fast_rates_MOM} is only granted on a subset of the $L_2$-sphere centered in the oracle $f^*$ where the radius is proportional to the rate of convergence of the estimators. We are now in position to state our main results for the minmax MOM estimators.
\begin{Theorem}\label{main:theo_mom}
	Grant Assumptions~\ref{assum:lip_conv},~\ref{assum:convex},~\ref{assum:phi},~\ref{assum:moments} and~\ref{assum:fast_rates_MOM}. Let $S \geq  7|\mathcal{O}|/3 $,
	Then, with probability larger than $1- 2\exp(-S/504)$, for any regularization parameter $\lambda > C_{S,r}(\tilde A) / \phi(f^*)$, the estimator $\hat f_S^{\lambda}$ defined in Equation~\eqref{def:MOM} satisfies
	\begin{align*}
		& \phi(\hat f_S^{\lambda}-f^*)  \leq \eta(4+2/\tilde A) \phi(f^*) , \quad  \|\hat f_S^{\lambda}-f^{*}\|_{L_2} \leq (4+6\tilde{A}) \lambda\frac{ \phi(f^*)}{\sqrt{C_{S,r}(\tilde A)}} 
	\end{align*}
\end{Theorem}	
It is also possible to use the Lepski's method to get an adaptive estimator as the one in Theorem~\ref{theo_lepski}. For the sake of brevity, we do not present this result here. There is a tradeoff between confidence and accuracy and an optimal choice of $S$  would be $S \asymp \tilde{r}(\tilde{A}) N $. In that case, $C_{S,r}(\tilde A) \asymp \tilde{r}(\tilde{A})$. For this value of $S$, the optimal $\lambda$ is $\tilde{r}^2(\tilde{A})/ \phi(f^*)$ and we would obtain $ \|\hat f_S^{\lambda}-f^{*}\|_{L_2}^2 \lesssim C(\tilde A) \tilde r(\tilde A) $.  With $S \asymp \tilde{r}(\tilde{A}) N $  and $\lambda=
\asymp \tilde{r}^2(\tilde{A})/ \phi(f^*)$, we recover the same result as the one in the sub-Gaussian setting as long as Rademacher complexity and Gaussian-mean width are equivalent. We will see in Section~\ref{sec_svm} that it is the case for the precise example of RKHS associated to bounded kernel. Moreover, by construction, the estimator $\hat f_S^{\lambda}$ is robust to $3S/7$ outliers in the dataset.\\
Therefore, using minmax-MOM estimators, we have relaxed two strong Assumptions 1) the i.i.d setting and 2) the sub-Gaussian Assumption on the class $F-f^*$. Properly calibrated minmax-MOM estimators are not affected if the number of outliers is less than \emph{number of observations} $\times$ \emph{square of the optimal rate in the i.i.d setup} (when $S \asymp \tilde{r}(\tilde{A}) N $ and $r(A) \asymp \tilde r (A)$).


\section{Applications} \label{sec_applications}

Our results are very general and may be applied to various examples. To do so, it is necessary to:
\begin{itemize}
	\setlength{\itemsep}{1pt}
	\setlength{\parskip}{0pt}
	\setlength{\parsep}{0pt}
	\item Verify Assumptions~\ref{assum:lip_conv},~\ref{assum:convex} and \ref{assum:phi}.
	\item If the RERM is studied, check Assumption~\ref{ass:sub-gauss} and compute the Gaussian-mean-width $w\big(F \cap B_{\eta(4+2(A^*)^{-1}) \phi_j}^{\phi}(f^*)  \cap(f^* r B_{L_2}) \big)$ to deduce $r(A)$ for every $A>0$.
	\item If the minmax MOM-estimators is considerer, compute the Rademacher complexity to deduce $\tilde r(A)$ for every $A>0$.
	\item Find $A$ satisfying the local Bersntein condition (the $L_2$-radius depends on the estimator we consider).
\end{itemize}
As an illustration, we study in the sequel RERM and minmax MOM-estimators for linear estimators in $\bR^p$ regularized by the elastic net and for regularized kernel methods. It turns out that the sub-Gaussian assumption over the class $F-f^*$ is not required by using the reproducing property of RKHS. Instead we develop another general analysis to study RERM in RKHS associated with bounded kernel (see Section~\ref{sec_special_case_rkhs}).

\subsection{Application to Elastic net with Huber loss function} \label{sec_elastic_net}

In~\cite{zou2005regularization}, the authors noticed that the performance of the LASSO is not as good as the one of Rigde regression when the variables are highly correlated. Theoretically, it is now known that the covariance matrix of the design $X$ must satisfy the Restricted Eigenvalue condition to obtain fast rates of convergence for the LASSO \cite{bellec2018slope,bickel2009simultaneous}. To bypass this limitation, the authors introduced in~\cite{zou2005regularization} the Elastic net regularization.

\paragraph{Regularized Empirical Risk Minimizers}
Let $F$ be the class of linear functionals in $\bR^p$, $F = \{ \inr{\cdot,t}, \; t \in \bR^p  \}$ which satisfies Assumption~\ref{assum:convex}. Let $(X_i,Y_i)_{i=1}^N$ be random variables valued in $\bR^p \times \cY$. As the \emph{oracle} is denoted $f^*$, we introduce $t^*$ such that $f^*(\cdot)= \inr{t^*,\cdot}$. Let $\alpha \in [0,1]$, for any $t$ in $\bR^p$, the elastic net penalization is defined as
\begin{equation} \label{elastic_net}
\phi(t) = (1-\alpha) \|t\|_1 + \alpha \|t\|_{2}^2 \enspace,
\end{equation} 
where $\|t\|_1 = \sum_{i=1}^p |t_i|$ and $\|t\|^2_2 = \sum_{i=1}^p t_i^2$. For $\alpha = 1$ and $\alpha = 0$ we recover respectively the ridge and the Lasso penalizations (these cases will not be studied in the sequel). Clearly $\phi$ defined in Equation~\eqref{elastic_net} satisfies Assumption~\ref{assum:phi} with $\eta = 2$.
Let $\bar{\ell}^{\delta}$ be the  huber loss function with parameter $\delta > 0$ (which is $\delta$-Lipschiz), the estimator RERM is defined as
\begin{equation} \label{def_en}
\hat t_{\lambda}^{\delta,\alpha} \in \argmin_{t \in \bR^p} \frac1{N} \sum_{i=1}^N  \bar{\ell}^{\delta} \big( \inr{X_i,t},Y_i   \big) + \lambda \big( (1-\alpha) \|t\|_1 + \alpha \|t\|_{2}^2 \big) \enspace.
\end{equation}

Theorems~\ref{thm_erm_conv2} and~\ref{theo_lepski} require the computation of the Gaussian mean-width $w \big(F \cap B_{\rho}^{\phi}(t^*)  \cap (f^*+r B_{L_2})  \big)$ for $r,\rho >0$. To do so, let us assume that the design $X$ is isotropic i.e for all $t \in \bR^p$, $\bE \inr{X,t}^2_{\bR^p} = \|t\|_{2}^2$. It means that the $L_2(\mu)$ norm
coincides with the natural Euclidean structure on the space $\ell_2^p$. Thus, for all $\rho , r >0 $, under the isotropic assumption, we have
\begin{equation}
w \big(F \cap B_{\rho}^{\phi}(t^*)  \cap (f^*+r B_{L_2})  \big) = w(B_{\rho}^{\phi}(0) \cap r B_2^p) = \bE \sup_{ t \in \bR: \; (1-\alpha) \|t\|_1 + \alpha \|t\|^2_2 \leq \rho, \; \|t\|_2 \leq r} \inr{\textbf{G},t}_{\bR^p} \enspace,
\end{equation}
where $\textbf{G}$ is a standard Gaussian random vector in $\bR^p$ and $B_l^p$ denotes the unit ball in $(\bR^p,\|\cdot\|_l)$, for $l \geq 0$. Let $\alpha \in (0,1)$. We have,
\begin{equation} \label{w_elastic}
w(B_{\rho}^{\phi}(0) \cap r B_{L_2}) \leq \min\bigg(  w\big( \frac{\rho}{1-\alpha} B_1^p \cap r B_2^p\big), w\big( \min(r,\sqrt{\frac{\rho}{\alpha}}) B_2^p \big) \bigg ) \enspace.
\end{equation}  
Let us introduce
\begin{align*}
	& r_1^* = \inf \big\{ r> 0: \;   64\delta BA^*w\bigg( \frac{(8+4/A^*)\phi(f^*)}{1-\alpha} B_1^p \cap r B_2^p\bigg)  \leq  \sqrt{N}r^2   \big\} \enspace. \\
	& r_2^* = \inf \big\{ r> 0: \;   64\delta BA^*w\bigg( \min \big(r,\sqrt{\frac{(8+4/A^*)\phi(f^*)}{\alpha}}) B_2^p\big) \bigg)  \leq  \sqrt{N}r^2   \big\} \enspace.
\end{align*}
From Equation~\eqref{w_elastic} and the definition of $r^*$ it is clear that $r^* \leq \min(r_1^*,r_2^*)$. Using  the computations of $w\big( \rho B_1^p \cap r B_2^p\big)$ for all $r,\rho >0 $ presented in~\cite{lecue2017regularization}, it follows that
\begin{align*}
	(r_1^*)^2 =  \left\{
	\begin{array}{ll}
		\frac{(8+4/A^*)\phi(f^*)}{1-\alpha} \sqrt{\frac{64\delta BA^*}{ \pmb{N}}\log \bigg(  \frac{e \pmb p(1-\alpha)}{\sqrt{\pmb{N}}(8+4/A^*)\phi(f^*)}\bigg)} & \mbox{if } \frac{(8+4/A^*)^2\phi^2(f^*)\pmb N}{(1-\alpha)^264\delta BA^*} \leq \pmb p^2 \\
		\frac{64\delta BA^* \pmb p}{\pmb N}	& \mbox{if }\frac{ (8+4/A^*)^2\phi^2(f^*)\pmb N}{(1-\alpha)^264\delta BA^*} \geq \pmb p^2 \end{array}
	\right.
\end{align*}

\begin{align*}
	(r_2^*)^2 =  \left\{
	\begin{array}{ll}
		\frac{64\delta BA^* \pmb p}{\pmb N}  & \mbox{if } \pmb N \geq \frac{64\delta BA^*\alpha \pmb p}{(8+4/A^*)\phi(f^*)}  \\
		\sqrt{\frac{64\delta B (8+4/A^*) \phi(f^*) \pmb p }{\alpha \pmb N}}& \mbox{if } \pmb N \leq \frac{64\delta BA^*\alpha \pmb p}{(8+4/A^*)\phi(f^*)} \end{array}
	\right.
\end{align*}
For the sake of presentation, the dependence with respect to the dimension and the sample size is presented in bold. Since $r^* \leq \min(r_1^*,r_2^*)$, it is clear that $r^*$ captures the best situation between the LASSO (complexity parameter $r_1^*$) and the Ridge regression (complexity parameter $r_2^*$). \\

To apply Theorems~\ref{thm_erm_conv2} and~\ref{theo_lepski}, it remains to verify the local Bernstein condition. Results on the local Bernstein Assumption (see Assumptions~\ref{assum:fast_rates} and~\ref{assum:fast_rates_MOM}) can be found in~\cite{chinot} for the quantile and Huber losses for regression problems and for the logistic and the Hinge loss for classification. For the sake of brevity, we only present the results for the Huber loss function with parameter $\delta >0$  (absolute loss function will be studied in Section~\ref{sec_svm}). Note that $\delta$ must be of the order of a constant. Let us introduce the following assumption.
\begin{Assumption} \label{ass:ber_huber} Let $r,\rho,\varepsilon >0$.
	\begin{itemize}
		\setlength{\itemsep}{1pt}
		\setlength{\parskip}{0pt}
		\setlength{\parsep}{0pt}
		\item a) There exists $C'>0$ such that, for all $ f\in F$ such that $\|f-f^*\|_{L_2} = r$ and $\phi(f-f^*) \leq \rho$, $\|f-f^{*}\|_{L_{2+\varepsilon}} \leq C'\|f-f^{*}\|_{L_2}$. 
		\item b) Let $C'$ be the constant defined above. There exists $\gamma>0$ such that, for all $x\in\cX$ and for all $z$ in $\mathbb{R}$ such that $ |z-f^{*}(x) | \leq (\sqrt{2}C')^{(2+\varepsilon)/\varepsilon} r $, we have $ F_{Y|X=x}(z+\delta)-F_{Y|X=x}(z- \delta)  \geq \gamma$, where $F_{Y|X=x}$ is the conditional cumulative function of $Y$ given $X=x$.
	\end{itemize}
\end{Assumption}
When the class $F-f^*$ is $1$-sub-Gaussian, it is clear that the point a) of Assumption~\ref{ass:ber_huber} holds with an absolute constant $C'$ for $\varepsilon = 2$ (see theorem 1.1.5 in \cite{chafai2012interactions}). For the point b), if $Y = \inr{t,X} + W$, where $W$ is a symmetric random variable independent from $X$ and $t \in \bR^p$, we have $t^* = t$. In this case, the point b) holds if $F_W(\delta - 2(C')^2r ) - F_W(2(C')^2r - \delta) \geq \gamma$, where $F_W$ denotes the cdf of $W$. It simply means that the noise puts enough mass around 0. In particular, point b) holds when $W$ is Cauchy. In this case, $Y$ is not integrable and yet we are able to verify the Bernstein condition and derive fast rates of convergence. 
\begin{Theorem}[\cite{chinot}] \label{huber_loss}
	Grant Assumptions \ref{ass:ber_huber} (with parameter $r$, $\rho$ and $\gamma$). Then, for all $f\in F$ satisfying $\|f-f^*\|_{L_2} = r$ and $\phi(f-f^*) \leq \rho$, $\|f-f^*\|_{L_2}^2\leqslant (4/\gamma) P\mathcal{L}_f$.	
\end{Theorem}
Note that in~\cite{chinot}, the proof holds for any $f$ in $F$ such that $\|f-f^*\|_{L_2} = r$. The proof of Theorem~\ref{huber_loss} is exaclty the same as the one in \cite{chinot} with simple modifications taking into account the new localization with respect to the regularization. \\
We are now in position to state the main theorem for the elastic net procedure.
\begin{Theorem} \label{theo_elastic_net}
	Let $r^* = \min(r^*_1,r_2^*)$. Let $(X_i,Y_i)_{i=1}^N$ be i.i.d random variables distributed as $(X,Y)$ where $Y = \inr{X,t^*}+ W$, where $t^* \in \bR^p$, $X = (x_1,\cdots,x_p)$ is a sub-Gaussian random vector. Let us assume that the noise $W$ is a symmetric random variable independent from $X$ such that there exists $\gamma >0$ for which $F_W(\delta - 2(C')^2r^*) - F_W(2(C')^2r^* - \delta) \geq \gamma$ . Let $\lambda = (r^*)^2 / \phi(f^*)$. With probability larger than $1-2\exp\big(-  \frac{\gamma^2}{4(128)^2\delta^2}N (r^*)^2 \big)$, the estimator $\hat t_{\lambda}^{\delta,\alpha}$ associated with the Huber loss function defined in Equation~\eqref{def_en} satisfies
	\begin{align*}
		\| \hat t_{\lambda}^{\delta,\alpha} - t^{*}\|_{2} \leq  (4+24/\gamma) r^*  \quad  \phi(\hat t_{\lambda}^{\delta,\alpha} - f^{*}) \leq  (8+\gamma) \phi(t^*)   \quad \mbox{and} \quad  P\cL_{\hat t_{\lambda}^{\delta,\alpha}} \leq (4+3\gamma/4) (r^*)^2\enspace.  
	\end{align*}
\end{Theorem}
In Theorem~\ref{theo_elastic_net} we set $\lambda = (r^*)^2/\phi(f^*)$ which is evidently unknown. However it is possible to use Theorem~\ref{theo_lepski} to get an adaptive estimator for the Elastic net achieving the same rates. When $1-\alpha$ is close to $1$ that is when the penalization $\ell_1$ is dominant we have $r^* = r^*_1$ and we recover the result for the Lasso (see~\cite{lecue2017regularization}). When $\alpha$ is close to $1$ the elastic net is almost equivalent to ridge regression and $r^* = r^*_2$. We recover the results for the ridge regression.\\
In Theorem~\ref{theo_elastic_net} it is not clear if there exists $\gamma$ such that $F_W(\delta - 2(C')^2r^*) - F_W(2(C')^2r^* - \delta) \geq \gamma$. It turns out that this condition is very weak. It simply means that the noise $W$ puts enough mass around $0$. For instance let $W$ be a standard Cauchy distribution. The condition $F_W(\delta - 2(C')^2r^*) - F_W(2(C')^2r^* - \delta) \geq \gamma$ can be rewritten as $\delta - 2(C')^2r^* \geq \tan(\gamma \pi/2)$. If $r^* \leq 1$ we can take $\gamma =1$ and $\delta = 4(C')^2 +1$. The condition $r^* \leq 1$ means that enough data are given to the statistican which corresponds to interesting learning problems. Consequently, even for non-integrable noise such as a Cauchy distribution we are able to derive fast rates of convergence.


\paragraph{Minmax MOM-estimators}

Now, let us turn to the robust minmax MOM-estimator associated with the Huber loss function for the elastic net procedure defined as
\begin{equation} \label{def_en_MOM}
\hat t_{\lambda,S}^{\delta,\alpha} \in \argmin_{t \in \bR^p}  \sup_{\tilde t \in \bR^p}  MOM_S \big( \ell^{\delta}_t - \ell^{\delta}_{\tilde t} \big) + \lambda \big( (1-\alpha) (\|t\|_1 - \|\tilde t\|_1) + \alpha ( \|t\|_{2}^2 - \|\tilde t\|_2^2) \big)
\end{equation}
where $\ell^{\delta}$ denotes the Huber loss function with parameter $\delta$. 
To study these estimators, is necessary to compute the rademacher complexity given in the definition of $\tilde r(\cdot)$. From Theorem 1.6 in~\cite{mendelson2017multiplier}, it is possible to link Rademacher complexity and Gaussian mean-width for the Elastic-net regularization as long as $X$ is isotropic (i.e for all $t$ in $\bR^p$,  $\bE \inr{X,t}^2 = \|t\|_2^2$ ) and satisfies
\begin{equation} \label{cond_rad_en}
\forall 1 \leq q \leq c_1\log(p), 1\leq i \leq p, \quad  \|\inr{X,e_i}\|_{L_q} \leq c_1 \sqrt{q} \enspace,
\end{equation}
for $c_1,c_2>0$ two absolute constants and where $(e_i)_{i=1}^p$ denotes the canonical basis of $\bR^p$. Since any a real valued random variable $Z$ is $L_0$-sub-Gaussian if and only if for all $q \geq1$, $\|Z\|_{L_q} \leq c_3 L_0 \sqrt{q}$, for $c_3>0$ an absolute constant, the condition~\eqref{cond_rad_en} imposes “$c_1\log(p)$ sub-Gaussian moments“ on the design $X$. From Theorem 1.6 in~\cite{mendelson2017multiplier}, if condition~\eqref{cond_rad_en} holds, we get $\tilde r(\tilde A) \leq c_4 r(\tilde{A}) $ for $c_4>0$ an aboslute constant and the following theorem holds:
\begin{Theorem} \label{theo_elastic_net_mom}
	Let $\tilde r = c_4\min(r_1^*,r_2^*)$. Let $(X,Y)$ be a random variable such that $Y = \inr{X,t^*}+ W$, where $t^* \in \bR^p$ and $W$ a symmetric random variable independent from $X$ such that there exists $\gamma >0$ with $F_W(\delta - 2(C')^2\tilde r) - F_W(2(C')^2\tilde r - \delta) \geq \gamma$. $X$ is assumed to be an isotropic random vector satisfying condition~\eqref{cond_rad_en}. Assume that $(X_i,Y_i)_{i \in \cI}$ are independent and distributed as $(X,Y)$. Let $S\geq 7|\cO|/3$.  With probability larger than $1 - \exp(-S/504)$, the estimators $\hat t_{\lambda,S}^{\delta,\alpha}$ defined in \eqref{def_en_MOM} with
	\begin{equation*}
		\lambda = \frac{\max \big( (\tilde r)^2,\frac{5588\delta^2}{\gamma^2} \frac{S}{N}  \big)}{\phi(t^*)}
	\end{equation*}
	satisfies 
	\begin{align*}
		\| \hat t_{\lambda,S}^{\delta,\alpha} & - f^{*}\|_{2}^2 \leq  (8+\gamma)^2 \max \bigg( (\tilde r)^2, \frac{5588\delta^2}{\gamma^2} \frac{S}{N}  \bigg)  \quad   \phi(\hat t_{\lambda,S}^{\delta,\alpha} - f^{*}) \leq  (4+3\gamma/4) \phi(t^*) \enspace. 
	\end{align*}
\end{Theorem}
When $S \lesssim N (\tilde r)^2$, Theorem~\ref{theo_elastic_net_mom} improves Theorem~\ref{theo_elastic_net} by relaxing the sub-Gaussian Assumption. Moreover, for $S \asymp N (\tilde r)^2$ up to $3N (\tilde r)^2/7$ outliers can be present in the dataset without affecting the error rate. Note also that it is possible to adapt the estimator in a data-driven way to the best $S$ and $\lambda$ by using a Lepski's adaptation as we have done in Theorem~\ref{theo_lepski}. 

\begin{Remark}
	In Theorems~\ref{theo_elastic_net} and~\ref{theo_elastic_net_mom}, we assumed that the design $X$ is isotropic. This assumption is only used for the computation of the Gaussian mean-with of the intersection of the $\ell_1$ ball with the $\ell_2$ ball. Using the recent work from~\cite{bellec2017localized} it is possible to extend the result for more general covariance matrices. 
\end{Remark}

\subsection{Application to RKHS} \label{sec_svm}

In this section, we consider regularization methods in some general Reproducing Kernel Hilbert Space (RKHS) (cf. \cite{steinwart2008support} for a specific analysis on RKHS). The regularization function $\phi(\cdot)$ is defined as $ \phi(\cdot) = \|\cdot\|_{\cH_K}^2$ where $ \|\cdot\|_{\cH_K}$ is the norm in the space $\cH_K$ associated to a kernel $K$. 
This section is inspired from the work in~\cite{pierre2019estimation}. The authors established convergence rates when $\phi(\cdot) = \|\cdot\|_{\cH_k}$ and $F = RB_{\cH_K}$, for $R>0$, for classification problems under a much stronger global Margin assumption. We improve their work in many aspects 1) heavy-tailed noise can be handled, 2) the margin assumption is replaced by the weaker local Bernstein condition, 3) we can analyse the regularization $\phi(\cdot) = \|\cdot\|_{\cH_K}^2$  and 4) there is no restriction on the class $F = \cH_K$, we do not restrict $F$ to be a regularization ball in $\cH_K$. \\
Using Theorems~\ref{main:theo_mom} we derive explicit bounds on the error rates depending on $\|f^*\|_{\cH_K}$ for the minimax-MOM estimators. For the RERM, we could use Theorem~\ref{thm_erm_conv2}. However, it turns out that the sub-Gaussian Assumption~\ref{ass:sub-gauss} on the class $F-f^*$ is complicated to verify for RKHS and the application of Theorem~\ref{thm_erm_conv2} may be tricky. Instead, we derive another analysis where no sub-Gaussian assumption is required. In the precise example of RKHS, our homogeneity argument implies that we can restrict ourselves to a bounded class of functions. As a consequence, we can use concentration tools such as Talagrand's inequality instead of results from the sub-Gaussian theory. Nothing has to be assumed on the design $X$. \\

We are given $N$ pairs $(X_i,Y_i)_{i=1}^N$ of random variables where the $X_i$'s take their values in some measurable space $\cX$ and $Y_i \in \cY$ where $\cY = \{-1,1\}$ for binary classification problems and $\cY = \bR$ for regression problems.  We introduce a kernel $K: \cX \times \cX \mapsto \bR$ measuring a similarity between elements of $\cX$ i.e $K(x_1,x_2)$ is small if $x_1,x_2 \in \cX$ are “similar".  The main idea of kernel methods is to transport the design data $X_i$'s from the set $\cX$ to a certain Hilbert space via the application $x \mapsto K(x,\cdot) := K_x(\cdot)$ and construct a statistical procedure in this "transported" and structured space. The kernel $K$ is used to generate an Hilbert space known as Reproducing Kernel Hilbert Space (RKHS). Recall that if $K$ is a positive definite function i.e for all $n \in \bN^* $, $x_1,\cdots,x_n \in \cX$ and $c_1,\cdots,c_n \in \bR$, $\sum_{i=1}^n \sum_{j=1}^n c_i c_j K(x_i,x_j) \geq 0$, then by Mercer's theorem there exists an orthonormal basis $(\phi_i)_{i=1}^{\infty}$ of $L_2(\mu)$ such that $\mu \times \mu$ almost surely, $K(x,y) = \sum_{i=1}^{\infty} \lambda_i \phi_i(x) \phi_i(y)$, where $(\lambda)_{i=1}^{\infty}$ is the sequence of eigenvalues (arranged in a non-increasing order) of $T_K$ and $\phi_i$ is the eigenvector corresponding to $\lambda_i$ where 
\begin{align} \label{def_tk}
	\nonumber T_K: \; & L_2(\mu) \to L_2(\mu) \\
	& (T_Kf)(x) = \int K(x,y) f(y) d\mu (y)  \enspace.
\end{align}
The Reproducing Kernel Hilbert Space $\cH_K$ is the set of all functions of the form $\sum_{i=1}^{\infty} a_i K(x_i,\cdot)$ where $x_i \in \cX$ and $a_i \in \bR$ converging in $L_2(\mu)$ endowed with the inner product
\begin{equation*}
	\inr{\sum_{i=1}^{\infty} a_i K(x_i,\cdot),\sum_{i=1}^{\infty} b_i K(y_i,\cdot)} = \sum_{i,j=1}^{\infty} a_i b_j K(x_i,y_i) \enspace.
\end{equation*}
An alternative way to define a RKHS is via the feature map $\Phi: \cX \mapsto \ell_2$ such that $\Phi(x) = \big(\sqrt{\lambda_i}\phi_i(x)\big)_{i=1}^{\infty} $. Since $(\Phi_k)_{k=1}^{\infty}$ is an orthogonal basis of $\cH_K$, it is easy to see that the unit ball of $\cH_K$ can be expressed as 
\begin{equation}
B_{\cH_K} = \{ f_{\beta}(\cdot) = \inr{\beta,\Phi(\cdot)}_{\ell_2}, \; \| \beta \|_{2}  \leq 1  \}  \enspace,
\end{equation}
where $\inr{\cdot,\cdot}_{\ell_2}$ is the standard inner product in the  Hilbert space $\ell_2$. 
In other words, the feature map $\Phi$ can the used to define an isometry between the two Hilbert spaces $\cH_K$ and $\ell_2$. \\
The RKHS $\cH_K$ is therefore a convex class of functions from $\cX$ to $\bR$ that can be used as a learning class $F$. Let the \emph{oracle} $f^*$ be defined as 
\begin{equation*}
	f^* \in \argmin_{f \in \cH_K} \bE [  \bar{\ell}(f(X),Y) ] \enspace.
\end{equation*}
Let $f$ be in $\cH_K$, by the reproducing property and Cauchy-Schwarz we have for all $x,y$ in $\cX$
\begin{align} \label{bounded:RKHS}
	|f(x)-f(y)| = \inr{f, K_x -K_y} \leq \|f\|_{\cH_K}  \| K_x - K_y \|_{\cH_K}   \enspace.
\end{align}
From Equation~\eqref{bounded:RKHS}, it is clear that the norm of a function in the RKHS controls how fast the function varies over $\cX$ with respect to the geometry defined by the kernel (Lipschitz with constant $\|f \|_{\cH_K})$. As a consequence the norm of regularization $\| \cdot \|_{\cH_K}$ is related with its degree of smoothness w.r.t. the metric defined by the kernel on $\cX$. Let $\bar \ell$ be any loss function satisfying Assumption~\ref{assum:lip_conv}, the estimators $\hat{f}_{\lambda}^{\phi}$ and $\hat{f}_{\lambda,S}^{\phi}$ defined respectively in Equation~\eqref{def_erm2} and \eqref{def:MOM}  are given by
\begin{equation} \label{def_svm}
\hat{f}_{\lambda}^{\phi} = \argmin_{f \in \cH_K} \frac{1}{N} \sum_{i=1}^{N} \bar \ell(f(X_i),Y_i) + \lambda \|f\|_{\cH_K}^2
\end{equation}
and
\begin{equation} \label{def_svm_mom}
\hat{f}_{\lambda,S}^{\phi} = \argmin_{f \in \cH_K} \sup_{g \in \cH_K} MOM_S(\ell_f-\ell_g) + \lambda\big(  \|f\|_{\cH_K}^2-\|g\|_{\cH_K}^2 \big).
\end{equation}
It is clear that $\phi(\cdot) = \|\cdot\|_{\cH_K}^2$ verifies Assumption~\ref{assum:phi} with $\eta = 2$
We establish oracle inequalities for $\hat{f}_{\lambda}^{\phi} $ and $\hat{f}_{\lambda,S}^{\phi}$ respectively defined in Equation~\eqref{def_svm} and~\eqref{def_svm_mom} when the loss satisfies Assumption~\ref{assum:lip_conv}. In~\cite{mendelson2010regularization,meister2016optimal,wu2006learning,smale2007learning} for the quadratic loss function and~\cite{eberts2013optimal,farooq2017learning} for the pinball loss (which is Lipschitz), the authors establish error bounds for when the target $Y$ is assumed to satisfy $Y \in [-M,M]$ almost surely which is a really strong Assumption. Our analysis applies when the target $Y$ is unbounded and may even be heavy-tailed which is, as far as we know, a new result. In~\cite{caponnetto2007optimal} the authors do not assume that the target $Y$ is bounded. However, their analysis requires to control the Laplace transform of the noise $Y-f^*(X)$ (see Assumption 2 in~\cite{caponnetto2007optimal}). As a consequence they cannot consider heavy-tailed noise. In~\cite{eberts2013optimal,farooq2017learning} the authors are also interested in the approximation error of kernel methods and compare ourselves with their results is a complicated task. We obtain the same error rate as ~\cite{mendelson2010regularization,caponnetto2007optimal} when the eigenvalues of the integral operator $T_K$ satisfies $\lambda_n   \leq \beta n^{-1/p}$ for some $0 < p <1$ and $\beta>0$ an absolute constant when $Y$ may be \textbf{unbounded and heavy-tailed}. The value of $p$ is related with the smoothness of the space $\cH_K$. Different kinds of spectrum could be analysis. It would only change the computation of the complexity fixed-points. For the sake of simplicity we only focuse on this example as it has been studied in~\cite{caponnetto2007optimal,mendelson2010regularization} for instance. \\

\subsubsection{New general analysis for the RERM} \label{sec_special_case_rkhs}
Since every RKHS are convex, Assumption~\ref{assum:convex} holds. Therefore, when the loss function satisfies Assumption~\ref{assum:lip_conv}, to use Theorem~\ref{thm_erm_conv2} it is necessary to verify Assumptions~\ref{ass:sub-gauss} and~\ref{assum:fast_rates}. However, it turns out that the sub-Gaussian Assumption on the class $F-f^*$ cannot be verfied in practice except for very precise Kernels. Our analysis (see Section~\ref{sec_proof}) requires the sub-Gaussian Assumption to show that with an exponentially large probability for all $f$ in $F$ such that $\|f-f^*\|_{L_2} \leq r(A^*)$ and $\phi(f-f^*) \leq \eta(2+2/A^*)  \phi(f^*)$: 
\begin{equation} \label{sto_arg}
\big | (P-P_N) \cL_f \big|  \leq \frac{r^2(A^*)}{2A^*} \enspace,
\end{equation}
where $A^*$ satisfies Assumption~\ref{assum:fast_rates} and $r(\cdot)$ is the complexity parameter defined in Definition~\ref{def:function_r}. However, when $F = \cH_K$ we have $\{ f \in F:  \phi(f-f^*) \leq \eta(2+2/A^*) \phi(f^*)  \} =  \{ f \in \cH_K : \|f-f^*\|_{\cH_K} \leq 2\sqrt{1+1/A^*} \|f^*\|_{\cH_K}    \}$. Morever, from the reproducible property, for all $x \in \cX$ and all $f$ in $\cH_K$ such that $\|f-f^*\|_{\cH_K} \leq 2\sqrt{1+1/A^*} \|f^*\|_{\cH_K} $ we have
\begin{align*}
	|f(x)-f^*(x)| = \inr{f-f^*,K_x}_{\cH_K}& \leq \|f-f^*\|_{\cH_K} \|K_x\|_{\cH_K} \\
	&=  \|f-f^*\|_{\cH_K} \sqrt{K(x,x)} \leq 2\sqrt{(1+1/A^*)\|K\|_{\infty}} \|f^*\|_{\cH_K} 
\end{align*}
Therefore, when $F = \cH_K$, for $K$ a bounded Kernel,  the control of \eqref{sto_arg} is over a  bounded class of functions. As a consequence, the sub-Gaussian Assumption is no longer necessary. Instead we develop another analysis based of the Bousquet's version of Talagrand's inquality~\cite{bousquet2002bennett}. Since no sub-Gaussian assumption is required we use another complexity parameter where the Rademacher complexity replaces the Gaussian mean-width.
\begin{equation} \label{def:function_r_rkhs}
\bar r(A) = \inf \bigg\{r>0, \quad  \bE \sup_{ \substack{f \in F: \|f-f^*\|_{L_2} \leq r, \\  \enspace \|f-f^*\|_{\cH_K} \leq 2\sqrt{2+1/A} \|f^*\|_{\cH_K}  }   } \enspace \sum_{i=1}^N \sigma_i (f-f^*)(X_i)  \leq \frac{N r^2}{64AL} \bigg\}
\end{equation}
We also adapt the local Bernstein assumption to the Definition~\eqref{def:function_r_rkhs}.
\begin{Assumption}\label{assum:fast_rates_RKHS} There exists a constant $\bar A \geq 1$ such that for all $f\in \cH_K$ if \\$\norm{f-f^*}_{L_2} =   2L\sqrt{(2+1/\bar A)\|K\|_{\infty}} \|f^*\|_{\cH_K} \bar  r(\bar A)$ and $ \|f-f^*\|_{\cH_K} \leq 2\sqrt{2+1/\bar A} \|f^*\|_{\cH_K}  $ then $\|f-f^{*}\|_{L_2}^2\leqslant \bar A P\mathcal{L}_f $. 
\end{Assumption}

\begin{Theorem} \label{thm:erm_RKHS}
	Let $(X_i,Y_i)_{i=1}^N$ be i.i.d random variables with commom distribution $P$. Let $\ell$ be a loss function satisfying Assumption~\ref{assum:lip_conv} with $L \geq1$. Let $\cH_K$ be a RKHS associated to a bounded Kernel $K$. Grant Assumption~\ref{assum:fast_rates_RKHS} such that $\bar A \geq 1$. Let $U = 2L\sqrt{(2+1/\bar A)\|K\|_{\infty}}$.
	With probability larger than
	\begin{equation*}
		1-2\exp\bigg(-   \frac{N \bar r^2(\bar A)}{64(L\bar A)^2} \bigg)
	\end{equation*}
	for all regularization parameters $\lambda \geq \lambda_0 = \max(1,U\|f^*\|_{\cH_K}) \bar  r^2(\bar A) / \|f^*\|_{\cH_K}^2 $ the estimators $\hat{f}^{\phi}_{\lambda}$ defined in Equation~\eqref{def_svm} satisfies
	\begin{align*}
	&	\|\hat{f}^{\phi}_{\lambda} - f^{*}\|_{L_2} \leq ( 4+6\bar A ) \lambda \frac{ \|f^*\|_{\cH_K}^2 }{ \max(1, \sqrt {U\|f^*\|_{\cH_K}})\bar r(\bar A)} \\
		& \mbox{and} \quad  \| \hat{f}^{\phi}_{\lambda} - f^{*} \|_{\cH_K}\leq  (8+4/\bar A) \|f^*\|_{\cH_K}.  
	\end{align*}
\end{Theorem}
The proof can be find in Section~\ref{sec:proof_RKHS}. Theorem~\ref{thm:erm_RKHS} is similar to Theorem~\ref{thm_erm_conv2} for RKHS when the sub-Gaussian assumption is relaxed. By taking $\lambda = \lambda_0$ we get 
\begin{align*}
	\|\hat{f}^{\phi}_{\lambda} - f^{*}\|_{L_2} \leq  (4+6\bar A) \max(1, \sqrt{U\|f^*\|_{\cH_K}}) \bar r(\bar A) \enspace.
\end{align*}
When $\|f^*\|_{\cH_K} \leq M$, we obtain the same bounds as the one in Theorem~\ref{thm_erm_conv2} (up to a constant depending on $\bar A$ and $\|K\|_{\infty}$) and a Lespki's procedure as in Theorem~\ref{theo_lepski} yields to an adaptive estimator. Note that the assumption that $\|K\|_{\infty} < \infty$ is really weak since any continuous kernel on a compact space is bounded. Moreover many results in RKHS are derived for the Gaussian Kernel with is bounded by $1$, \cite{farooq2017learning,steinwart2008support}. 

\subsubsection{Explicit bounds for the ERM and the minmax MOM estimators}
To obtain explicit bounds in Theorems~\ref{main:theo_mom} and~\ref{thm:erm_RKHS} it is necessary to calculate the complexity parameters $\bar r(\bar A)$ and $\tilde r(\tilde A)$. To do so, we have to compute the Rademacher complexity of the set $\{f \in \cH_K: \; \|f-f^*\|_{\cH_K}^2 \leq \rho, \; \|f-f^*\|_{L_2}\leq r  \}$ for any $\rho, r > 0$. From Theorem 2.1 in~\cite{mendelson2003performance}, if $K$ is a bounded kernel, then for all $\rho,r >0$
\begin{equation*}
	\mathbb{E} {\sup_{f \in \cH_K  \cap (f^*+ rB_{L_2} \cap \rho B_{\cH_K} )} \frac{1}{\sqrt N}\bigg |{\sum_{i  =1}^N \sigma_i (f-f^*)(X_i)} \bigg |}  \leq \sqrt{2} \|K\|_{\infty} \bigg( \sum_{k=1}^{\infty} \big(  \rho^2 \lambda_k \wedge r^2 \big) \bigg)^{1/2}
\end{equation*}

\begin{Remark}
	Since the feature map $\Phi$ defines an isometry between $\cH_K$ and $\ell_2$, the computation of the Gaussian mean-width of the set $\{f \in \cH_K: \; \|f-f^*\|_{\cH_K}^2 \leq \rho, \; \|f-f^*\|_{L_2}\leq r  \}$ is equivalent to the computation of the Gaussian mean-width of an ellipsoid in $\ell_2$. Consequently, it is easy to show that Rademacher complexity and Gaussian mean-width (and thus $\bar r(A)$ and $r(A)$ ) are equivalent. 
\end{Remark}
In the case where the eigenvalues $\lambda_k \leq \beta k^{-1/p}$ for all $k \in \bN^*$ and $0<p< 1$, where $\beta >0$ is an absolute constant and $\rho/r \geq 1$, straightforward computations give
\begin{equation*}
	\bigg( \sum_{k=1}^{\infty} \big(  \rho^2 \lambda_k \wedge r^2 \big) \bigg)^{1/2} \leq \beta \frac{\rho^p}{r^{p-1}} 
\end{equation*} 
It follows that for any bounded kernel $K$ such that the eigenvalues assoicated to $T_K$ satisfy $\lambda_k \leq \beta k^{-1/p}$ for all $k \in \bN^*$ and $0<p< 1$ and $A>0$
\begin{align*}
	& \tilde r^2(A) =  C(A,\beta ,L , p) \frac{\|f^*\|^{(2p)/(p+1)}_{\cH_K}}{N^{1/(p+1)}} = 6  \bar r^2( A)
	\\ 
	& \mbox{where} \quad C(A,\beta ,L , p) = \big( 384 A \beta L \big)^{2/(p+1)}  \big( 4(2+1/  A)  \big)^{2p/(p+1)}
\end{align*}

Now, let us turn to  Bernstein condition. We use the results from~\cite{chinot} where the local Bernstein condition has been extensively studied for many convex and Lipschitz loss functions. In Section~\ref{sec_elastic_net} we studied the Huber loss function. Here, we consider the absolute loss (which is the quantile loss for $\tau=1/2$). Let us present the Assumptions required to study the Bernstein condition for the quantile loss function. 
\begin{Assumption} \label{ass:ber_quantile} Let $r,\rho,\varepsilon >0$.
	\begin{itemize}
		\setlength{\itemsep}{1pt}
		\setlength{\parskip}{0pt}
		\setlength{\parsep}{0pt}
		\item a) There exists $C'>0$ such that for all $ f\in F$ such that $\|f-f^*\|_{L_2} = r$ and $\|f-f^*\|_{\cH_K} \leq \rho$, $\|f-f^{*}\|_{L_{2+\varepsilon}} \leq C'\|f-f^{*}\|_{L_2}$ 
		\item b) Let $C'$ be the constant defined above. There exists $\alpha>0$ such that, for all $x\in\cX$ and for all $z$ in $\mathbb{R}$ such that $ |z-f^{*}(x) | \leq (\sqrt{2}C')^{(2+\varepsilon)/\varepsilon} r $, we have $ f_{Y|X=x}(z) \geq \gamma$, where $f_{Y|X=x}$ is the conditional density function of $Y$ given $X=x$.
	\end{itemize}
\end{Assumption}
Assumption~\ref{ass:ber_quantile} and~\ref{ass:ber_huber} are very similar. When $Y = f^*(X) +W$, for $f^*$ in $\cH_K$ and $W$ is a symmetric noise, condition b) simply means that the noise $W$ puts enough mass around $0$.

\begin{Theorem} [\cite{chinot}]\label{quantile_loss}
	Grant Assumptions \ref{ass:ber_quantile} (with parameter $r,\rho$ and $\gamma$). Then, for all $f\in F$ satisfying $\|f-f^*\|_{L_2} = r$ and $\|f-f^*\|_{\cH_K} \leq \rho$, $\|f-f^*\|_{L_2}^2\leqslant (4/\gamma) P\mathcal{L}_f$.	
\end{Theorem}

For kernel methods, the point a) of Assumption~\ref{ass:ber_quantile} is a $L_{2+ \varepsilon} / L_2$-norm equivalence which is only required in the ball defined by the norm in the RKHS. Let $f$ in $F$ such that $\|f-f^*\|_{\cH_K} \leq \rho$ and $\|f-f^*\|_{L_2} = r$, we have
\begin{equation*}
	\|f-f^{*}\|_{L_{2+\varepsilon}}^{2+\varepsilon} = \int (f(x)-f^*(x))^{2+\varepsilon} dP_X(x) \leq  (\rho \|K\|_{\infty} )^{\varepsilon} \|f-f^*\|_{L_2}^{2}
\end{equation*}
Since $\|f-f^*\|_{L_2} = r$, it follows that
\begin{equation*}
	\|f-f^{*}\|_{L_{2+\varepsilon}} \leq  \bigg( \frac{\rho\|K\|_{\infty} }{r} \bigg)^{\varepsilon/(2+\varepsilon)}  \|f-f^*\|_{L_2}. 
\end{equation*} 
Therefore, the point a) holds with $C' =  (\rho\|K\|_{\infty}/r)^{\varepsilon/(2+\varepsilon)}$. Let us turn to the point b). From the fact that $C'= (\rho\|K\|_{\infty}/r)^{\varepsilon/(2+\varepsilon)}$, we have $\sqrt{2C'}^{(2+\varepsilon)/\varepsilon} r = 2^{(2+\varepsilon)/2\varepsilon} \rho \|K\|_{\infty}$. For example, when $Y = g(X) + W$, where $g \in \cH_K: \cX \mapsto \bR$ and $W$ is symetric and independent from $X$, it is easy to see that $f^* = g$. In this case the second point of Assumption~\ref{ass:ber_quantile} can be rewritten as $f_W(z) \geq \gamma$ for all $z \in \bR$ such that $|z| \leq 2^{(2+\varepsilon)/2\varepsilon} \rho \|K\|_{\infty}$, where $f_W$ denotes the density function of $W$. It simply means that the noise puts enough mass around $0$.


We are now in position to state our main Theorems in a RKHS associated with a bounded kernel when the absolute loss function is considered for the RERM and the minmax MOM estimators. 		
\begin{Theorem} \label{thm_rkhs_quantile}
	Let $\cX$ be some measurable space and $K: \; \cX \times \cX \mapsto \bR$ be a positive definite bounded kernel where $\cH_K$ denote its associated RKHS. Let $(\lambda_k)_{k=1}^{\infty}$ be the sequence of eigenvalues associated to $T_K$ in $L_2(\mu)$ such that $\lambda_k \leq \beta k^{-1/p}$ for all $k \in \bN^*$ and $0<p< 1$, where $\beta >0$ is an absolute constant. 
	For any $x \in \cX$, let $f_{Y|X=x}$ denote the conditional density function of $Y$ given $X=x$. Let us assume that there exists $\gamma>0$ such that, for all $x\in\cX$ and for all $z$ in $\mathbb{R}$ such that $ |z-f^{*}(x) | \leq 2\sqrt{8+\gamma} \|f^*\|_{\cH_K}\|K\|_{\infty} $, we have $ f_{Y|X=x}(z) \geq \gamma$. Let $(X_i,Y_i)_{i=1}^N$ be i.i.d random variables distributed as $(X,Y)$. Then with probability larger than 
	\begin{equation*}
		1 - \exp\bigg(  - \frac{\gamma C(4/\gamma,\beta,1,p)}{256}N^{p/(p+1)} \|f^*\|_{\cH_K}^{2p/(p+1)}\bigg)\enspace,
	\end{equation*}
	when
	\begin{equation*}
		\lambda =   C(4/\gamma,\beta,1,p) \max(1,(8+\gamma) \|K\|_{\infty} \|f^*\|_{\cH_K}) \frac{\|f^*\|_{\cH_K}^{2/(p+1)}}{N^{1/(1+p)}} \enspace,
	\end{equation*}
	the estimator $\hat{f}^{\phi}_{\lambda}$ associated to the absolute loss function defined in Equation \eqref{def_svm} satisfies
	\begin{align*}
		& \|\hat{f}^{\phi}_{\lambda} - f^* \|_{L_2}^2 \leq (4+3/(2\gamma)) C(4/\gamma,\beta,1,p) \max(1,(8+\gamma) \|K\|_{\infty} \|f^*\|_{\cH_K})\frac{\|f^*\|_{\cH_K}^{2/(p+1)}}{N^{1/(1+p)}}  \\
		&  \mbox{and} \quad \|\hat{f}^{\phi}_{\lambda} - f^*\|_{\cH_K} \leq (8+\gamma)\|f^*\|_{\cH_K}
	\end{align*}
\end{Theorem}
The error rate in Theorem~\ref{thm_rkhs_quantile} is the same as in~\cite{mendelson2010regularization}. However \textbf{our analysis do not require that the target $Y$ is bounded. It can even be heavy-tailed}. Note also that nothing is assumed on the design $X$.  
\begin{Remark}
	When $Y = f^*(X) + W$,  where $W$ is a standard Cauchy distribution, the condition $ f_{Y|X=x}(z) \geq \gamma$ for  $z$ in $\mathbb{R}$ such that $ |z-f^{*}(x) | \leq 2\sqrt{8+\gamma} \|f^*\|_{\cH_K}\|K\|_{\infty} $ is satisfied as long as there exists $\gamma \in (0,1]$ such that
	\begin{equation*}
		\frac{1}{\pi \big( 1+ 4(8+\gamma) \|f^*\|_{\cH_K}^2\|K\|_{\infty}^2 \big) } \geq \gamma
	\end{equation*}
	which holds for  $\gamma = \min(1,1/(\pi(1+36\|f^*\|^2\|K\|^2_{\infty})))$. Consequenlty the analysis holds for heavy-tailed distribution.
\end{Remark}

Let us turn to the MOM-estimators. 
\begin{Theorem} \label{thm_rkhs_quantile_mom}
	Let $\cX$ be some measurable space and $K: \; \cX \times \cX \mapsto \bR$ be a positive definite bounded kernel where $\cH_K$ denote its associated RKHS. Let $(\lambda_k)_{k=1}^{\infty}$ be the sequence of eigenvalues associated to $T_K$ in $L_2(\mu)$ such that $\lambda_k \leq \beta k^{-1/p}$ for all $k \in \bN^*$ and $0<p< 1$, where $\beta>0$ is an absolute constant. 
	For any $x \in \cX$, let $f_{Y|X=x}$ denote the conditional density function of $Y$ given $X=x$. Let us assume that there exist $\gamma>0$ such that, for all $x\in\cX$ and for all $z$ in $\mathbb{R}$ such that $ |z-f^{*}(x) | \leq 2\sqrt{8+\gamma} \|f^*\|_{\cH_K} \|K\|_{\infty} $, we have $ f_{Y|X=x}(z) \geq \gamma$. Let us assume that $(X_i,Y_i)_{i \in \cI}$ are independent and distributed as $(X,Y)$. Let $S \geq 7|\cO|/3$. Let: 
	\begin{equation*}
		C_{S,N} = \max \bigg( 6C(4/\gamma,\beta,1,p) \frac{\|f^*\|^{(2p)/(p+1)}_{\cH_K}}{N^{1/(p+1)}},\frac{13888}{\gamma} \frac{S}{N} \bigg)
	\end{equation*}
	Then with probability larger than $1 - \exp(-S/504)$ when
	\begin{equation*}
		\lambda =  \frac{C_{S,N}}{\|f^*\|_{\cH_K^2}} \enspace,
	\end{equation*}
	the estimator $\hat{f}^{\phi}_{\lambda,S}$ associated to the absolute loss function defined in Equation \eqref{def_svm_mom} satisfies
	\begin{align*}
		& \|\hat{f}^{\phi}_{\lambda,S} - f^* \|_{L_2}^2 \leq (4+3/(2\gamma)) C_{S,N} \quad \mbox{and} \quad \|\hat{f}^{\phi}_{\lambda,S} - f^*\|_{\cH_K} \leq (8+\gamma) \|f^*\|_{\cH_K}
	\end{align*}
\end{Theorem}
When $S \lesssim N^{p/(p+1) } \|f^*\|^{(2p)/(p+1)}_{\cH_K}$ we recover the bounds from Theorem~\ref{thm_rkhs_quantile}. However for the minmax MOM-estimators, up to $3S/7$ outliers can contaminate the dataset without deteriorated the error rate.

\section{Conclusion}
We have presented two general results for the RERM and minmax-MOM estimators describing the statistical properties of regularization in learning theory. For those two estimators we do not assume that the regularization is a norm which is, as far as we know a new general result for Lipschitz and convex loss functions. Under the local Bernstein Assumption, we can obtain rates of convergence depending on $\phi(f^*)$.  Results for the RERM have been derived under the i.i.d and the sub-Gaussian Assumptions on the class $F-f^*$ while no concentration Assumption is required for minmax MOM-estimators. For MOM-estimators, a number of outliers smaller than \emph{square of the rate of convergence in a non-contaminated setting} $\times$ \emph{number of observations} does not deteriorate the learning procedure. We studied the particular example of SVM where no sub-Gaussian assumption on the class $F$ is required and when the target $Y$ may be heavy-tailed, widely improving the existing results in the literature. \\
There are a number of interesting directions in which this work can be extended. One relevant and closely related problem is to obtain sparsity bounds, i.e bounds depending on an underlying structure of the \emph{oracle} $f^*$ such as the sparsity or the rank of the \emph{oracle} $f^*$. It has been partially done (under a really strong Assumption) in \cite{pierre2019estimation,chinot2019robust} when the regularization function if a norm. However without this Assumption, the proofs no longer hold and a new analysis has to be developed. 

\appendix

\section{Proof of Theorems~\ref{thm_erm_conv2},~\ref{theo_lepski} RERM} \label{sec_proof}
	
	In the remaining of the proof we shall use repeatedly the following notations
	\begin{gather*}
		A = A^*,   \quad\theta = \frac{1}{2A}, \quad \delta = \frac{2}{A} + 3 \quad \gamma =  \frac{2}{A} + 2 \enspace.
	\end{gather*}
	
	\subsection{Proof Theorem~\ref{thm_erm_conv2}}
	Proof of Theorem~\ref{thm_erm_conv2} is split into two parts. First, we identify an event onto which the statistical behavior of the regularized estimator $\hat{f}_{\lambda} := \hat{f}^{\phi}_{\lambda}$ can be controled using only deterministic arguments. Then, we prove that this event holds with a probability at least as large as the one in \eqref{eq:proba}. Let us define $\rho^* = (2 + \gamma)\eta \phi(f^*)$. We first introduce this event:
	\begin{equation*}
		\Omega := \left\{ 
		\mbox{for all } f\in F\cap (f^*+ r^*B_{L_2}) \cap B_{\rho^*}^{\phi}(f^*) , \quad \big|(P-P_N)\cL_f\big|\leq \theta  (r^*)^2 \right\}
	\end{equation*} 
	where we recall that $r^* = r(A^*)$ and $B_{\rho^*}^{\phi}(f^*) =  \{ f \in F: \; \phi(f-f^*) \leq \rho^*  \} $.
	\begin{Lemma} \label{lemma_cont}
		Let $\lambda \geq (r^*)^2/\phi(f^*)$, on the event $\Omega$ we have
		\begin{itemize}
			\item For all $f \in  F \backslash \cB_{\lambda}$,	 $P_N \cL_f^{\lambda} > 2(\theta+1)\lambda \phi(f^*)$
			\item For all $f \in  F \cap \cB_{\lambda}$,	 $P_N \cL_f^{\lambda} \geq -2(\theta+1) \lambda \phi(f^*)$
		\end{itemize}
	\end{Lemma}
	\begin{Proposition}\label{prop:algebra2}Let $\lambda \geq \lambda_0 := (r^*)^2/ \phi(f^*)$, on the event $\Omega$, one has 
		\begin{align*}
			\phi(\hat{f}_{\lambda}-f^{*}) \leq \rho^*,\quad
			\|\hat{f}_{\lambda}-f^{*}\|_{L_2} &\leq \lambda \frac{\delta \phi(f^*)}{(A^{-1}-\theta)r^*} 
		\end{align*}
	\end{Proposition} 
	
	\begin{proof}
		Let $\lambda \geq \lambda_0$, we denote $\mathcal B_{\lambda} =  \bigg(f^*+\big( \lambda \delta \phi(f^*) /((A^{-1}-\theta)r^*) \big)B_{L_2} \bigg) \cap B_{\rho^*}^{\phi}(f^*) $. We want to prove that $\hat f_{\lambda} \in \mathcal B_{\lambda}$.
		We recall that the regularized empirical excess loss function is defined for all $f\in F$ by
		\begin{equation*}
			P_N\cL_f^\lambda = P_N\cL_f + \lambda \big(\phi(f) - \phi(f^*)\big).
		\end{equation*}
		Since $\hat f_{\lambda}$ is such that $P_N\cL_{\hat f_{\lambda}}^\lambda\leq 0$, it is enough to prove that $P_N\cL_f^\lambda>0$ for all $f\in F \backslash \mathcal B_{\lambda}$ to get that $\hat f_{\lambda}\in  \mathcal B_{\lambda}$. In fact, for the adaptive procedure it will be necessary to use the results from Lemma~\ref{lemma_cont} which is equivalent (up to the choice of the constants) to show than $P_N\cL_f^\lambda>0$ for all $f\in F \backslash\mathcal B_{\lambda}$. From Lemma~\ref{lemma_cont} it follows immediately that $	\phi(\hat{f}_{\lambda}-f^{*}) \leq \rho^*$ and 
		$ \|\hat{f}_{\lambda}-f^{*}\|_{L_2} \leq \lambda \frac{\delta\phi(f^*)}{(A^{-1}-\theta)r^*} $
		
	\end{proof}
	
	\begin{proof}{\textbf{Lemma~\ref{lemma_cont}}}
		
		The proof follows from an homogeneity argument saying that if $P_N\cL_{f_0}^{\lambda}> 2(\theta +1 ) \lambda \phi(f^*) $ on the border of $ \mathcal B_{\lambda}$ then we also have $P_N\cL_f ^{\lambda}>  2(\theta +1 ) \lambda \phi(f^*)$ for all $f\in F$ outside $\mathcal B_{\lambda}$. Inside $ \mathcal B_{\lambda}$ the arguments are similar. \\
		Let $f$ in $F$ be outside of $\mathcal B_{\lambda}$. By convexity of $F$, there exists $f_0\in F$ and $\alpha>1$ such that $f-f^* = \alpha(f_0-f^*)$ and $f_0\in\partial \mathcal B_{\lambda}$ where we denote by $\partial \mathcal B_{\lambda}$ the border of $\mathcal B_{\lambda}$. By definition, we either have: 1) $\phi(f_0-f^*)= \rho^*$ and $\norm{f_0-f^*}_{L_2}\leq (\lambda \delta\phi(f^*))/((A^{-1}-\theta)r^*)$ in that case, $\alpha$ is such that  $ 1 \leq  \alpha \leq  \phi(f-f^*)/ \rho^* $ (see Lemma~\ref{lem_homo} in Section~\ref{app_supp}) or 2) $\norm{f_0-f^*}_{L_2}= (\lambda \delta \phi(f^*))/((A^{-1}-\theta)r^*) $ and $\phi(f_0-f^*) \leq \rho^*$ and, in that case, $\alpha = \norm{f-f^*}_{L_2}/\big( (\lambda \delta\phi(f^*))/((A^{-1}-\theta)r^*) \big)$. We will treat the two cases independently.\\ 
		
		Let us first explain the role of the convexity of the loss function by writing down an homogeneity argument linking the empirical excess risk of $f$ to the one of $f_0$. For all $i \in \{1,\cdots,N \}$, let $\psi_i: \mathbb R \rightarrow \mathbb R $ be defined for all $u\in \R$ by 
		\begin{equation}\label{eq:fct_psi}
		\psi_i(u) = \bar{\ell} (u + f^{*}(X_i), Y_i) - \bar{\ell} (f^{*}(X_i), Y_i).
		\end{equation}
		The functions $\psi_i$ are such that $\psi_i(0) = 0$, they are convex because $\bar{\ell}$ is, in particular $\alpha\psi_i(u) \leq \psi_i(\alpha u)$ for all $u\in\mathbb R$ and $\alpha \geq 1$ and $\psi_i(f(X_i) - f^{*}(X_i) )=  \bar{\ell} (f(X_i), Y_i) - \bar{\ell} (f^{*}(X_i), Y_i) $ so that the following holds:
		\begin{align} \label{conv_arg}
			\nonumber P_N \cL_f & = \frac{1}{N} \sum_{i=1}^{N}  \psi_i \big( f(X_i)- f^{*}(X_i) \big)= \frac{1}{N} \sum_{i=1}^{N}   \psi_i(\alpha( f_0(X_i)- f^{*}(X_i) ))\\
			&\geq \frac{\alpha}{N} \sum_{i=1}^{N}   \psi_i(( f_0(X_i)- f^{*}(X_i))) = \alpha P_N \cL_{f_0}.
		\end{align}
		For the  regularization part the same homogeneity arguments holds. 	
		\begin{align*}
			\phi(f) - \phi(f^{*})  = \phi \big(f^* + \alpha(f_0 - f^*) \big) - \phi(f^{*}) \geq \alpha \big( \phi(f_0) - \phi(f^*)  \big)
		\end{align*}
		where we used Lemma~\ref{lem_conv} (see Section~\ref{app_supp}). Therefore 
		\begin{equation*}
			P_N \cL_f^{\lambda} \geq \alpha 	P_N \cL_{f_0}^{\lambda} 
		\end{equation*}
		Let us now place ourselves on the event $\Omega$ up to the end of the proof and let $f_0\in F \cap \partial \mathcal B_{\lambda}$. We explore two cases depending on the localization of $f_0$ on the border of $\mathcal B_{\lambda}$: 1) $\phi(f_0-f^*)= \rho^*$ and $\norm{f_0-f^*}_{L_2}\leq (\lambda \delta \phi(f^*))/((A^{-1}-\theta)r^*)$ which is the case where the regularization part helps to show that  $P_N\cL_{f_0}^\lambda>  2(\theta +1 ) \lambda \phi(f^*)$ or 2) $\norm{f_0-f^*}_{L_2}= (\lambda \delta \phi(f^*))/((A^{-1}-\theta)r^*)$ and $\phi(f_0-f^*)\leq \rho^*$ which is where the Bernstein's condition helps. 
		\input{dessin_1}
		We consider the first case which is when $\phi(f_0-f^*)= \rho^*$ and $\norm{f_0-f^*}_{L_2}\leq  (\lambda \delta \phi(f^*))/((A^{-1}-\theta)r^*)$. There are two cases, either $\norm{f_0-f^*}_{L_2}\leq r^*$ or $\norm{f_0-f^*}_{L_2}\geq r^*$. In both cases, from the fact that $\phi(f_0-f^*) \leq \eta \big( \phi(f_0) + \phi(f^*)  \big)$ we have $\phi(f_0) - \phi(f^*) \geq \gamma \phi(f^*)$. If $\norm{f_0-f^*}_{L_2}\leq r^*$, on $\Omega$ we have $	|(P-P_N)\cL_{f_0}|\leq \theta (r^*)^2$ and we get
		\begin{align*}
			P_N \cL_{f}^\lambda = P_N \cL_{f}+ \lambda\left(\phi(f) - \phi(f^*)\right)& \geq \alpha \big(P_N\cL_{f_0} + \lambda \gamma \phi(f^*)\ \big)  \geq \alpha \big(-\theta (r^*)^2 +  \gamma \lambda \phi(f^*)\ \big) \\
			& \geq (-\theta + \gamma) \lambda \phi(f^*) > 2(\theta + 1 ) \lambda \phi(f^*)
		\end{align*}
		where we used the facts that $\lambda \geq (r^*)^2/  \phi(f^*) $ and $P\cL_{f_0} \geq 0$ . If $r^* \leq  \norm{f_0-f^*}_{L_2} \leq \lambda \delta  \phi(f^*) / ((A^{-1}-\theta)r^*)$
		we use the same projection trick. Let $\alpha_1 = \|f_0-f^*\|_{L_2} / r^*$ and set $f_1$ in $F$ be such that $f_0 - f^* =  \alpha_1 (f_1- f^*)$. We have $\|f_1-f^*\|_{L_2} = r^*$ and $\phi(f_1-f^*) \leq \rho^* $. Therefore on $\Omega$ we have 
		\begin{equation*}
			P_N \cL_{f}^\lambda  \geq \alpha \big(P_N\cL_{f_0} + \gamma \lambda \phi(f^*)\ \big) \geq \alpha \big( \alpha_1 P_N \cL_{f_1} + \gamma \lambda \phi(f^*)\ \big) \geq   \gamma \lambda \phi(f^*) > 2 (\theta + 1 ) \lambda \phi(f^*)
		\end{equation*}
		Since,  on $\Omega$, $P_N \cL_{f_1} \geq P\cL_{f_1} - \theta (r^{*})^2 \geq A^{-1} \|f_1-f^*\|_{L_2} - \theta (r^{*})^2= (A^{-1}  - \theta) (r^{*})^2 > 0$ where we used Assumption~\ref{assum:fast_rates}. \\

		We now turn to the second case where $\norm{f_0-f^*}_{L_2} = \lambda \delta \phi(f^*) /((A^{-1}-\theta)r^*)$ and $ \phi(f_0-f^*) \leq \rho^*$. Remember that in this case $\alpha = \norm{f-f^*}_{L_2}/\big( (\lambda \delta\phi(f^*))/((A^{-1}-\theta)r^*) \big)$. The regularization part no longer helps. However, by the Bernstein Assumption~\ref{assum:fast_rates} and using the same projection trick we get 
		\begin{align*}
			P_N \cL_{f} \geq \frac{\|f-f^{*}\|_{L_2}}{(\lambda \delta \phi(f^*))/((A^{-1}-\theta)r^*)}  P_N \cL_{f_0}& \geq \frac{\|f-f^{*}\|_{L_2}}{(\lambda \delta \phi(f^*))/((A^{-1}-\theta)r^*)} \frac{\|f_0-f^{*}\|_{L_2}}{r^*} P_N\cL_{f_1} \\
			& \geq \frac{\|f-f^{*}\|_{L_2}}{r^*} (A^{-1}-\theta )  (r^*)^2 
		\end{align*}where $f_1$ is such that $f_0 - f^* = \big( \|f_0-f^{*}\|_{L_2}/(r^*)  \big)(f_1- f^*)$. We have $\|f_1-f^*\|_{L_2} = r^*$ and $\phi(f_1-f^*) \leq \rho^* $. Since $\norm{f-f^*}_{L_2} \geq \lambda \delta \phi(f^*) /((A^{-1}-\theta)r^*)$, we finally get  
		\begin{align*}
			&P_N \cL_{f}^{\lambda }  \geq  \frac{\|f-f^{*}\|_{L_2}}{r^*} (A^{-1}-\theta) (r^*)^2 - \lambda \phi(f^*) \geq (\delta -1) \lambda \phi(f^*) > 2(\theta + 1 ) \lambda \phi(f^*)
		\end{align*}
		We conclude the proof by studying $P_N \cL_f^{\lambda}$ for $f \in F \cap \cB_{\lambda}$. One more time there are two cases, either $\|f-f^* \|_{L_2} \leq r^*$ or $\|f-f^* \|_{L_2} \geq r^*$. In the first case, since $P\cL_{f_0}$, on $\Omega$ we get that 
		\begin{equation*}
			P_N \cL_f^{\lambda} \geq -\theta (r^*)^2 - \lambda \phi(f^*) \geq -(\theta +1)\lambda \phi(f^*)
		\end{equation*} 
		For $\|f-f^* \|_{L_2} \geq r^*$ using the projection trick, there exists $\alpha \geq 1$ such that $P_N \cL_f \geq \alpha P_N\cL_{f_0}$
		where $f_0$ satisfies $\|f_0-f^*\|_{L_2} = r^*$ and $\phi(f_0-f^*) \leq \rho^*$. Therefore on $\Omega $, using Assumption~\ref{assum:fast_rates}, we get $P_N \cL_f \geq \alpha (A^{-1}-\theta) (r^*)^2 \geq -\theta \lambda \phi(f^*)$. Finally in that case
		\begin{equation*}
			P_N \cL_f^{\lambda} \geq -(\theta +1)\lambda \phi(f^*)
		\end{equation*} 
	\end{proof}
	
	Next, we prove that $\Omega$ holds with large probability. To that end, we use the results from \cite{pierre2019estimation}. 
	\begin{Lemma} \cite{pierre2019estimation} \label{lem:subgauss}
		Assume that Assumption~\ref{assum:lip_conv} and Assumption~\ref{ass:sub-gauss} hold. Let $F^\prime\subset F$ then for every $u>0$, with probability at least $1-2\exp(-u^2)$
		\begin{equation*}
			\sup_{f,g\in F^\prime}\left|(P-P_N)(\cL_f-\cL_g)\right|\leq \frac{16LB}{\sqrt{N}} \left(w(F^\prime) +   u d_{L_2}(F^\prime)\right)
		\end{equation*} where $d_{L_2}$ is the $L_2$ metric, $d_{L_2}(F^\prime)$ is the $L_2$ diameter of $F^\prime$.
	\end{Lemma}
	It follows from Lemma~\ref{lem:subgauss} that for any $u>0$, with probability larger that $1-2\exp(-u^2)$,
	\begin{align*}
		&\sup_{f \in F \cap  (f^{*} + r^* B_{L_2})  \cap B_{\rho^*}^{\phi}(f^*) } \big | (P-P_N)\cL_f  \big|  \leq \sup_{f,g \in  F \cap  (f^{*} + r^* B_{L_2})  \cap B_{\rho^*}^{\phi}(f^*) } \big | (P-P_N)(\cL_f-\cL_g)  \big| \\
		& \leq \frac{16LB}{\sqrt{N}} \bigg(   w \big(F \cap  (f^{*} + r^* B_{L_2})  \cap B_{\rho^*}^{\phi}(f^*) \big) + ud_{L_2} \big(F \cap  (f^{*} + r^* B_{L_2})  \cap B_{\rho^*}^{\phi}(f^*) \big)   \bigg) \enspace.
	\end{align*}
	We have $d_{L_2} \big(F \cap  (f^{*} + r^* B_{L_2})  \cap B_{\rho^*}^{\phi}(f^*) \big) \leq r^*$ and $w \big(F \cap  (f^{*} + r^* B_{L_2})  \cap B_{\rho^*}^{\phi}(f^*) \big)  = w\big(F \cap  r^* B_{L_2} \cap B_{\rho^*}^{\phi}(0) \big)$, By definition of the complexity parameter (see Equation~\eqref{def:function_r}), for $u = \theta \sqrt{N} r^*/(32LB) $, with probability at least 
	\begin{equation}
	1-2\exp\big(-\theta^2N (r^*)^2 /(32^2 L^2B^2 ) \big)
	\end{equation} 
	for every $f$ in $F\cap(f^*+ r^*B_{L_2})  \cap B_{\rho^*}^{\phi} (f^*)$,
	\begin{equation}
	\big | (P-P_N) \cL_f \big|  \leq \theta  (r^*)^2
	\end{equation}
	
	\subsection{Proof Theorem~\ref{theo_lepski}} \label{sec_proof_lepski}
	In this section we work on the event
	\begin{equation*}
		\tilde{\Omega}:= \left\{ 
		\mbox{for all } f\in F\cap \bigg(f^*+  \frac{2\delta}{A^{-1}-\theta} r^*B_{L_2}\bigg) \cap B_{\rho^*}^{\phi}(f^*) , \quad \big|(P-P_N)\cL_f\big|\leq \theta  (r^*)^2 \right\}
	\end{equation*} 
	Using the same proof as the one for $\Omega$, it easy to show that $\tilde{\Omega}$ holds with probability larger than
	\begin{equation*}
		1-2\exp\bigg( - \frac{ \big( \theta (A^{-1}-\theta) \big) ^2 N (r^*)^2}{ (64 LB\delta)^2} \bigg)
	\end{equation*}
	Note that $\Omega \subset \tilde{\Omega}$ and then Lemma~\ref{lemma_cont} still holds. \\
	Let us assume that $(\lambda_j)_{j=0}^J = (r_j^2/\phi_j)_{j=0}^J$ is non increasing. From the choice of  $(\phi_j)_{j=0}^J$, there exists $\tilde{k}$ such that $\phi_{\tilde{k}} \leq \phi(f^*)\leq 2\phi_{\tilde{k}}$. Note that if $(\lambda_j)_{j=0}^J$ is non decreasing, it is enough to use the same proof with $\tilde{k}$ such that $(1/2)\phi_{\tilde{k}} \leq \phi(f^*)\leq \phi_{\tilde{k}}$. \\
	Moreover, from Lemma~\ref{lemma_cont}, for all $\lambda \geq \lambda_0$, $T_\lambda(f^*) = - P_N \cL_{\hat{f}_{\lambda}}^{\lambda} \leq (\theta +1) \lambda \phi(f^*) \leq 2(\theta+1) \lambda \phi_{\tilde{k}}$. Since $ \phi_{\tilde{k}} \leq \phi(f^*)$ it follows that $\lambda_{\tilde{k}} \geq \lambda_0 $. And finally 
	\begin{equation} \label{exc_loss}
	P_N\cL_{\hat{f}_{\lambda}}^{\lambda} \leq 2(\theta+1) \phi_{\tilde k} \lambda_{\tilde{k}} \leq  2(\theta+1) \phi_{k} \lambda_{k} \mbox{   for all  } k \geq \tilde{k}
	\end{equation}
	From the definition of $k^*$ and Equation~\eqref{exc_loss} it follows that $k^* \leq \tilde k$ and thus, $\tilde{f} \in \hat{R}_{\tilde{k}}$. As a consequence, $P_N \cL_{\tilde{f}}^{\lambda_{\tilde{k}}} \leq T_{\lambda_{\tilde{k}}}(\tilde f)$ and we get 
	\begin{equation*}
		P_N \cL_{\tilde{f}}^{\lambda_{\tilde{k}}} \leq 2(\theta + 1) \lambda_{\tilde{k}} \phi_{\tilde{k}} \leq 2(\theta + 1) \lambda_{\tilde{k}} \phi(f^*)
	\end{equation*}
	From Lemma~\ref{lemma_cont} it follows that $\tilde{f}$  satisfies $\| \tilde{f} -f^* \|_{L_2}  \leq   \lambda_{\tilde{k}} \delta \phi(f^*) /((A^{-1}-\theta)r^*) \leq   2 \lambda_{\tilde{k}} \delta \phi_{\tilde k} /((A^{-1}-\theta)r^*) \leq  \big(2\delta /(A^{-1}-\theta) \big) r^* $  and $\phi(\tilde f -f^*) \leq \eta(2+\gamma) \phi(f^*)$. \\
	
	We finish this section by showing a \emph{oracle inequality} for $\tilde{f}$. From the fact that $\| \tilde{f} -f^* \|_{L_2}  \leq \big(2\delta /(A^{-1}-\theta) \big) r^* $  and $\phi(\tilde f -f^*) \leq \eta(2+\gamma) \phi(f^*)$, it follows, on $\tilde{\Omega}$ that $(P-P_N)\cL_{\tilde{f}} \leq \theta (r^*)^2$. For all $\lambda > 0$ 
	\begin{align*}
		P\cL_{\tilde{f}} = P_N \cL_{\tilde{f}}  + (P-P_N)\cL_{\tilde{f}} \leq P_N \cL_{\tilde{f}}^{\lambda} + \lambda \big(\phi(f^*)- \phi(\tilde f) \big) + \theta (r^*)^2 \leq   P_N \cL_{\tilde{f}}^{\lambda} + \lambda \phi(f^*) + \theta (r^*)^2 \enspace.
	\end{align*} In particular for $\lambda = \lambda_{\tilde{k}}$ one has $P_N \cL_{\tilde{f}}^{\lambda_{\tilde k}} \leq 2(\theta + 2) \phi_{\tilde{k}} \lambda_{\tilde k} \leq 2(\theta +1) (r^*)^2$ and $\lambda_{\tilde k} \phi(f^*) \leq 2 (r^*)^2$. Finally 
	\begin{equation*}
		P \cL_{\tilde{f}} \leq (4+ 3\theta)(r^*)^2
	\end{equation*}

	\section{Proof Theorem \ref{main:theo_mom} minmax MOM estimators}
	
	Let $\tilde r$ and $C_{S,r}$ design respectively $\tilde{r} (\tilde A)$ and $C_{s,r} (\tilde{A})$. Moreover, all along the proof, the following notations will be used repeatedly.
	\begin{gather*}
		A = \tilde{A},   \quad\theta = \frac{1}{2A}, \quad \delta = \frac{2}{A} + 3 \quad \gamma =  \frac{2}{A} + 2 , \quad \mu = \frac{\theta}{ 192 L } \enspace.
	\end{gather*}
	The proof is divided into two parts. First, we identify an event where the minmax MOM estimators $\hat{f}_S^{\lambda} := \hat{f}_S$ is controlled. 
	Then, we prove that this event holds with large probability. 
	Let $S \geqslant 7|\cO| /3$, and
	\begin{equation*}
		C_{s,r} = \max \bigg(\frac{96L^2S}{\theta^2 N},\tilde{r}^2 \bigg) \quad \mbox{and} \quad \rho^* = \eta( 2+ \gamma) \phi(f^*)
	\end{equation*}
	Let $\mathcal{B}_{\lambda,S} = \{ f \in E:  \;  \|f-f^*\|_{L_2} \leq \frac{\delta}{A^{-1}- \theta} \frac{\lambda \phi(f^*) }{\sqrt{C_{s,r}}} \quad \mbox{and} \quad \phi(f^*-f^*) \leq \rho^*  \}$. Consider the following event
	\begin{equation} \label{omega_K}
	\Omega_S   =  \bigg \{  \forall f \in F\cap \sqrt{C_{S,r}} B_{L_2} \cap B_{\rho^*}^{\phi} (f^*),  \qquad
	\sum_{s=1}^SI\bigg(\bigg| (P_{B_s} - P)(\ell_f-\ell_{f^{*}}) \bigg| \leq \theta C_{s,r}\bigg)\geqslant \frac S2 \bigg \}\enspace.
	\end{equation}

	\subsection{Deterministic argument}
	
	\begin{Lemma}\label{lem:Basic}
		$\hat f_S \in\mathcal{B}_{\lambda,S}$ if the following inequalities holds
		\begin{align}\label{obj_proof} 
			& \sup_{f \in   F \backslash \mathcal{B}_{\lambda,s} } \quad MOM_{S}(\ell_{f^{*}}-\ell_f) + \lambda \big( \phi(f^{*})-\phi(f)  \big) \leq - 2(\theta+1) \lambda \phi(f^*)\enspace,\\
			& \sup_{f \in F\cap \mathcal{B}_{\lambda,S} } MOM_{S}(\ell_{f^{*}}-\ell_f)+ \lambda \big( \phi(f^{*})-\phi(f)  \big) \leq  (\theta+1) \lambda \phi(f^*)\enspace.\label{cond_gen}
		\end{align}
		
	\end{Lemma}
	\begin{proof}
		For any $f\in F$, denote by $S(f)=\sup_{g\in F} MOM_S(\ell_f-\ell_g) + \lambda \big( \phi(f)-\phi(g)  \big)$.
		If \eqref{obj_proof} holds, by homogeneity of $MOM_S$, any $f\in F \backslash \mathcal{B}_{\lambda,S}$ satisfies
		\begin{equation*}
			S(f)\geqslant  MOM_S(\ell_f-\ell_{f^*})+ \lambda \big( \phi(f)-\phi(f^{*})  \big)> 2(\theta+1) \lambda \phi(f^*) \enspace.\label{eq:Task2} 
		\end{equation*}
		On the other hand, if \eqref{cond_gen} and~\eqref{obj_proof} hold,
		\begin{align} \nonumber
			S(f^{*})= & \sup_{f\in F } MOM_S(\ell_{f^*}-\ell_f) + \lambda \big( \phi(f^{*})-\phi(f)  \big)  \leqslant   (\theta+1) \lambda \phi(f^*)\enspace.\label{eq:Task1}
		\end{align}
		Thus, by definition of $\hat{f}_S$ and \eqref{cond_gen},
		\[
		S(\hat{f}_S)\leqslant S(f^*)\leqslant  (\theta+1) \lambda \phi(f^*) \enspace.
		\]
		Therefore, if \eqref{obj_proof} and \eqref{cond_gen} hold,
		$\hat{f}_s \in \mathcal{B}_{\lambda,S}$.	
	\end{proof}  
	
	\begin{Lemma} \label{lemma_mom}
		For all $S \geq 7|\cO|/3$ and $\lambda \geq C_{S,r} / \phi(f^*)$, inequalities \eqref{obj_proof} and \eqref{cond_gen} holds on $\Omega_S$. 
	\end{Lemma}

	\begin{proof}
		The arguments are exaclty the same as the one in the proof of Lemma~\ref{lemma_cont}. For all functions $f \in F\backslash \mathcal{B}_{\lambda,S}$ and for each block $B_s$ there exist $\alpha \geq 1$ and $f_0 \in F$ in the border of $\mathcal{B}_{\lambda,S}  $ such that $P_{B_s} \cL_f \geq \alpha P_{B_s} \cL_{f_0} $. We present here only one case (the others are trivial applications of the arguments in the proof of Lemma~\ref{lemma_cont}). In the case where $\phi(f_0-f^*)= \rho^*$ and $\sqrt{C_{S,r}} \leq  \norm{f_0-f^*}_{L_2} \leq  (\lambda \delta \phi(f^*))/((A^{-1}-\theta)\sqrt{C_{S,r}})$.  We still have $\lambda \big( \phi(f_0) -\phi((f^*) \big) \geq \lambda \gamma \phi(f^*)$. Using the projection trick, there exists $\alpha_1 > 1$ such that on each block $B_s$, $P_{B_s} \cL_{f_0} \geq \alpha_1 P_{B_s} \cL_{f_1}$ for $f_1$ such that $\|f_1-f^*\|_{L_2} = \sqrt{C_{S,r}}$ and $\phi(f_1-f^*) \leq \rho^*$ and then, on the event $\Omega_S$, one more than $S/2$ blocks $B_s$ 
		\begin{equation} \label{mom_ex}
		P_{B_s} \cL_{f}^\lambda  \geq \alpha \big(P_{B_s}\cL_{f_0} + \gamma \lambda \phi(f^*)\ \big) \geq \alpha \big( \alpha_1 P_{B_s} \cL_{f_1} + \gamma \lambda \phi(f^*)\ \big) \geq   \gamma \lambda \phi(f^*) > 2 (\theta + 1 ) \lambda \phi(f^*)
		\end{equation}
		where we used the fact that on $\Omega_S$, there are at least $S/2$ blocks $B_s$ such that, $P_{B_s} \cL_{f_1} \geq P\cL_{f_1} - \theta C_{S,r} \geq A^{-1} \|f_1-f^*\|_{L_2}^2 - \theta C_{S,r}= (A^{-1}  - \theta) C_{S,r} > 0 $ and  Assumption~\ref{assum:fast_rates_MOM}. \\
		As Equation~\eqref{mom_ex} holds on more than $S/2$ blocks we get that 
		\begin{equation*}
			MOM_S(\ell_{f} - \ell_{f^*}) + \lambda \big( \phi(f) - \phi(f^*)  \big) \geq 2( \theta +1) \lambda \phi(f^*) 
		\end{equation*}
		From the same arguments as the one in the proof of Lemma~\ref{lemma_cont} we finally obtain 
		\begin{align*}
			& \sup_{f \in   F \backslash \mathcal{B}_{\lambda,S} } \quad MOM_{S}(\ell_{f^{*}}-\ell_f) + \lambda \big( \phi(f^{*})-\phi(f)  \big)< - 2(\theta +1) \lambda \phi(f^*)\enspace,\\
			& \sup_{f \in F \cap \cB_{\lambda,S}} MOM_{S}(\ell_{f^{*}}-\ell_f)+ \lambda \big( \phi(f^{*})-\phi(f)  \big) \leq  (\theta+1) \lambda \phi(f^*) 
		\end{align*}
		which concludes to proof. 
	\end{proof}
	\subsection{Control of the stochastic event} \label{sec:mom}
	Contrary to the deterministic argument, the control of the stochastic event is very different from the one for the RERM. 
	\begin{Proposition}\label{lem:close}
		Grant Assumptions~\ref{assum:lip_conv},~\ref{assum:convex},~\ref{assum:phi},~\ref{assum:moments} and~\ref{assum:fast_rates_MOM}. Let $S \geq 7|\cO|/3$. Then $\Omega_S$ holds with probability larger than $1-2\exp(- S/504)$.
	\end{Proposition}
	
	\begin{proof}
		Let $\mathcal F = \{  f \in F: \; \|f-f^*\|_{L_2} \leq \sqrt{C_{S,r}}, \; \phi(f-f^*) \leq \rho^*  \}$  and let $h(t) = I \{ t\geq 2 \} + (t-1) I \{1 \leq t \leq 2 \}$. This function satisfies $\forall t \in \mathbb{R}^+, \;   I\{ t\geq 2 \} \leq h(t) \leq I \{ t\geq 1 \}$. Let $W_s = ((X_i,Y_i))_{i \in B_s}$ and, for any $f\in\cF$, let $G_f(W_s) = (P_{B_s} - P)(\ell_f-\ell_{f^{*}})$. Let also $C_{S,r}=\max \bigg( 96L^2S/(\theta^2 N), \tilde{r}^2 \bigg) $. For any $f\in\cF$, let
		\begin{align*}
			z(f) & = \sum_{s =1}^S I \{|G_f(W_s)|\leq \theta C_{S,r}\}\enspace.
		\end{align*}
		Proposition \ref{lem:close} will be proved if $\bP \big(z(f) \geq S/2) \geqslant 1-e^{-S/504}$.
		Let $\mathcal{S}$ denote the set of indices of blocks which have not been corrupted by outliers, $\mathcal{S} = \{s \in \{1,\cdots,S \} : B_s \subset \mathcal{I}\}$. Basic algebraic manipulations show that
		\begin{multline}\label{eq:LBzf1}
			z(f) \geqslant |\mathcal{S}| - \sup_{f\in \mathcal F} \sum_{s \in \mathcal{S}} \bigg( h \big(2(\theta C_{S,r})^{-1} | G_f(W_s)| \big) - \mathbb{E} h\big(2(\theta C_{S,r})^{-1} | G_f(W_s)|\big) \bigg) \\
			-  \sum_{s \in \mathcal{S} } \mathbb{E}h\big(2(\theta C_{S,r})^{-1} | G_f(W_s)|\big) \enspace.
		\end{multline} 
		The last term in \eqref{eq:LBzf1} can be bounded from below since for all $f\in \cF$ and $s\in \cS$,
		\begin{align*}
			\mathbb{E}h\big(2(\theta C_{S,r})^{-1} | G_f(W_s)|\big) & \leqslant \mathbb{P} \bigg( |G_f(W_s)| \geq \frac{\theta C_{S,r}}{2} \bigg) \leqslant \frac{4\mathbb{E}G_f(W_s)^2}{(\theta C_{S,r})^2}   \\
			&  \leqslant  \frac{4S^2}{\theta^2 C_{S,r}^2N^2} \sum_{i \in B_s} \mathbb{E} [(\ell_f-\ell_{f^{*}})^2(X_i,Y_i)] \leq \frac{4L^2S}{\theta^2 C_{S,r}^2N}\|f-f^{*}\|^2_{L_2} \enspace.
		\end{align*}
		The last inequality follows from Assumption \ref{assum:moments}. 
		Since $\|f-f^{*}\|_{L_2} \leq \sqrt{C_{S,r}}$,
		\[
		\mathbb{E}h\big(2(\theta C_{S,r})^{-1} | G_f(W_s)|\big)\leqslant \frac{4L^2S}{\theta^2 C_{S,r}N}\enspace.
		\]
		As $C_{S,r}\geqslant 96L^2S/(\theta^2 N)$,  
		\begin{equation*}
			\mathbb{E}h\big(2(\theta C_{S,r})^{-1} | G_f(W_s)|\big)  \leq \frac{1}{24} \enspace.
		\end{equation*}
		Plugging this inequality in \eqref{eq:LBzf1} yields
		\begin{align}\label{res::1}
			z(f) \geq |\mathcal{S}|(1-\frac1{24}) -\sup_{ f \in \mathcal F
			} \sum_{s \in \mathcal{S}} \bigg( h\big(2(\theta C_{S,r})^{-1} | G_f(W_s)|\big) - \mathbb{E} h\big(2(\theta C_{S,r})^{-1} | G_f(W_s)|\big) \bigg) \enspace.
		\end{align}
		Using the Mc Diarmid's inequality, with probability larger than $1-\exp(- |\cS|/288 )$ we get 
		\begin{align*}
			\sup_{f \in \mathcal F} & \sum_{s \in \mathcal{S}} \bigg( h\big(2(\theta C_{S,r})^{-1} | G_f(W_s)|\big)  - \mathbb{E} h\big(2(\theta C_{S,r})^{-1} | G_f(W_s)|\big) \bigg)   \\
			& \leq  \frac{|\mathcal{S}|}{24} + \mathbb{E} \sup_{f \in \mathcal F}   \sum_{s \in \mathcal{S}} \bigg( h\big(2(\theta C_{S,r})^{-1} | G_f(W_s)|\big) - \mathbb{E}  h\big(2(\theta C_{S,r})^{-1} | G_f(W_s)|\big) \bigg)\enspace.
		\end{align*}
		By the symmetrization lemma, it follows that 
		\begin{multline*}
			\sup_{f \in \mathcal F}  \sum_{s \in \mathcal{S}} \bigg( h\big(2(\theta C_{S,r})^{-1} | G_f(W_s)|\big)  - \mathbb{E} h\big(2(\theta C_{S,r})^{-1} | G_f(W_s)|\big) \bigg) \\
			\leqslant \frac{|\mathcal{S}|}{24} + 2 \mathbb{E} \sup_{f \in \mathcal F }   \sum_{s \in \mathcal{S}} \sigma_k h\big(2(\theta C_{S,r})^{-1} | G_f(W_s)|\big) \enspace.
		\end{multline*}
		As $\phi$ is 1-Lipschitz with $\phi(0)=0$, the contraction Lemma from \cite{ledoux2013probability} and yields 
		\begin{align*}
			\sup_{f \in \mathcal F} \sum_{s \in \cS} \bigg( h\big(2(\theta C_{S,r})^{-1} | G_f(W_s)|\big) -  &\mathbb{E} h\big(2(\theta C_{S,r})^{-1} | G_f(W_s)|\big) \bigg)  \\
			& \leqslant \frac{|\mathcal{S}|}{24} + \frac{4}{\theta} \mathbb{E}  \sup_{f \in \mathcal F}  \quad  \sum_{s \in \mathcal{S}}  \sigma_s \frac{ G_f(W_s)}{C_{S,r}}  \\
			& = \frac{|\mathcal{S}|}{24} + \frac{4}{\theta} \mathbb{E} \sup_{f \in \mathcal F} \quad \sum_{s \in \mathcal{S}}  \sigma_s \frac{(P_{B_s}- P)(\ell_f-\ell_{f^{*}})}{C_{S,r}}  \\
		\end{align*}
		For any $s\in \cS$, let $(\sigma_i)_{i \in B_s}$ independent from $(\sigma_s)_{s \in \mathcal{S}}$, $(X_i)_{i \in \cI}$ and $(Y_i)_{i \in \cI}$. The vectors  $(\sigma_i \sigma_s (\ell_f-\ell_{f^{*}})(X_i,Y_i))_{i,f}$ and $(\sigma_i (\ell_f-\ell_{f^{*}})(X_i,Y_i))_{i,f}$ have the same distribution. Thus, by the symmetrization and contraction lemmas, with probability larger than $1-\exp(- |\cS|/288)$,
		\begin{align}
			\notag	\sup_{f \in \mathcal F } \sum_{s \in \mathcal{S}} \bigg( h\big(2C_{S,r}^{-1} | G_f(W_k)|\big)& -  \mathbb{E} h\big(2C_{S,r}^{-1} | G_f(W_s)|\big) \bigg)  \\
			\notag	&  \leq \frac{|\mathcal{S}|}{24}+ \frac{8}{\theta} \mathbb{E}  \sup_{f \in \mathcal F}  \quad \sum_{s \in \mathcal{S}} \frac{1}{|B_s|} \sum_{i \in B_s} \sigma_i \frac{ (\ell_f-\ell_{f^{*}})(X_i,Y_i)}{C_{S,r}}   \\ 
			\notag	& = \frac{|\mathcal{S}|}{24} + \frac{8S}{\theta N} \mathbb{E} \sup_{f \in \mathcal F} \quad \sum_{i \in \cup_{s \in \mathcal{S}} B_s}  \sigma_i  \frac{(\ell_f-\ell_{f^{*}})(X_i,Y_i)}{C_{S,r}} \\
			\label{eq:ConcB}	& \leq \frac{|\mathcal{S}|}{24} + \frac{8LS}{\theta N} \mathbb{E} \sup_{f \in \mathcal F} \bigg| \sum_{i \in \cup_{s \in \mathcal{S}} B_s}  \sigma_i \frac{(f-f^{*})(X_i)}{C_{S,r}} \bigg|\enspace.
		\end{align}
		Now either 1) $S \leq  \theta^2 \tilde{r}^2 N /(96L^2)$ or 2) $S >  \theta^2 \tilde{r}^2N /(96L^2)$. 
		Assume first that $S \leq  \theta^2 \tilde{r}^2 N /(96L^2)$, so $C_{S,r} = \tilde{r}^2$ and by definition of the complexity parameter 
		\begin{align*}
			\mathbb{E} \sup_{f \in \mathcal F } \bigg| & \sum_{i \in \cup_{s \in \mathcal{S}} B_s}  \sigma_i \frac{(f-f^{*})(X_i)}{C_{S,r}} \bigg| =  \mathbb{E} \sup_{f \in \mathcal F } \quad \frac{1}{\tilde{r}^2} \bigg| \sum_{i \in \cup_{s \in \mathcal{S}} B_s}  \sigma_i  (f-f^{*})(X_i)\bigg| \leq \frac{\mu |\cS| N}{S}\enspace.
		\end{align*}
		If $ S >  \theta^2 \tilde{r}^2N /(96L^2)$, $C_{S,r} = 96L^2S/( \theta^2 N)$. Then,
		\begin{align*}
			\mathbb{E} \sup_{f \in \mathcal F } \bigg| & \sum_{i \in \cup_{k \in \mathcal{S}} B_s}  \sigma_i \frac{(f-f^{*})(X_i)}{C_{S,r}}  \bigg |\\
			& \leq  	\mathbb{E}  \bigg[ \frac{1}{\tilde r^2} \sup_{f \in F \cap B_{\rho^*}^{\phi}(f^*) \cap \big(f^* +  \tilde r B_{L_2} \big) }  \bigg| \sum_{i \in \cup_{s \in \mathcal{S}} B_s}  \sigma_i (f-f^{*})(X_i) \bigg|  \\
			& \vee  \sup_{f \in F \cap B_{\rho^*}^{\phi}(f^*): \;  \tilde r \leq \|f-f^*\|_{L_2} \leq \sqrt{96L^2S/(\theta^2N)}}  \bigg| \sum_{i \in \cup_{s \in \mathcal{S}} B_s}  \sigma_i \frac{(f-f^{*})(X_i)}{96L^2S/(\theta^2N)} \bigg|   \bigg]
		\end{align*}
		By an homogeneity argument we obtain
		\begin{align*}
			&  \sup_{f \in F \cap B_{\rho^*}^{\phi}(f^*): \;  \tilde r \leq \|f-f^*\|_{L_2} \leq \sqrt{96L^2S/(\theta^2N)}}  \bigg| \sum_{i \in \cup_{s \in \mathcal{S}} B_s}  \sigma_i \frac{(f-f^{*})(X_i)}{96L^2S/(\theta^2N)} \bigg|   \bigg]\\
			& \leq 	\frac{1}{\tilde r}\sup_{f \in F \cap B_{\rho^*}^{\phi}(f^*): \;  \tilde r \leq \|f-f^*\|_{L_2} \leq \sqrt{96L^2S/(\theta^2N)}}  \bigg| \sum_{i \in \cup_{s \in \mathcal{S}} B_s}  \sigma_i \frac{(f-f^{*})(X_i)}{\|f-f^{*}\|} \bigg| \\
			& \leq \frac{1}{\tilde{r}^2} \sup_{f \in F \cap B_{\rho^*}^{\phi}(f^*): \;   \|f-f^*\|_{L_2}  = \tilde r}  \bigg| \sum_{i \in \cup_{s \in \mathcal{S}} B_s}  \sigma_i (f-f^{*})(X_i) \bigg| 
		\end{align*}
		Finally, in the second case  2) we also have
		\begin{align*}
			\mathbb{E} \sup_{f \in \mathcal F } \bigg| & \sum_{i \in \cup_{s \in \mathcal{S}} B_s}  \sigma_i \frac{(f-f^{*})(X_i)}{\max(\frac{4L^2S}{\alpha \theta^2 N}, \tilde r^2)} \bigg|  \leq \frac{\mu |\cS| N}{S}
		\end{align*}
		Plugging this bound in \eqref{eq:ConcB} yields, with probability larger than $1-e^{-|\cS|/288}$
		\[
		\sup_{f \in \mathcal F } \sum_{s \in \mathcal{S}} \bigg( h\big(2C_{S,r}^{-1} | G_f(W_s)|\big) -  \mathbb{E} h\big(2C_{S,r}^{-1} | G_f(W_s)|\big) \bigg)\leqslant |\mathcal{S}|\bigg(\frac{1}{24}+\frac{8L\mu}{\theta }\bigg)=\frac{|\cS|}{12}\enspace.
		\]
		Plugging this inequality into \eqref{res::1} shows that, with probability at least $1-e^{-|\cS|/288}$,
		\[
		z(f)\geqslant \frac{7|\cS|}8\enspace.
		\]
		As $S \geqslant 7|\cO|/3$, $|\cS|\geqslant S-|\cO|\geqslant 4S/7$, hence, $z(f)\geqslant S/2$ holds with probability at least $1-e^{-S/504}$.
	\end{proof}

	\section{Proof Theorem~\ref{thm:erm_RKHS}} \label{sec:proof_RKHS}
	As for the proof of Theorem~\ref{thm_erm_conv2} presented in Section~\ref{sec_proof} the proof is splitted into two parts. While we develop antoher stochastic argument the deterministic part from Proposition~\ref{prop:algebra2} is exaclty the same.\\ 
	In the example of RKHS, the sub-Gaussian Assumption is not necessary. Instead the tools from bounded class of function such as the Bousquet's inequality that we recall here can be used. 
	\begin{Theorem}[Theorem 2.6,~\cite{koltchinskii2011empirical}] \label{thm:bousquet}
		Let $\cF$ be a class of functions bounded by $M$. For all $t>0$, with probability larger than $1-\exp(-t)$
		\begin{equation}
		\sup_{f \in \cF} |(P_N-P)f| \leq \bE \sup_{f \in \cF} |(P_N-P)f| + \sqrt{2\frac{t}{N} \bigg( \sup_{f \in \cF} Pf^2 +  2M \bE \sup_{f \in \cF} |(P_N-P)f| \bigg)} + \frac{tM}{3N} 
		\end{equation}
	\end{Theorem}
	Let us define 
	\begin{align*}
		\Omega :=  \bigg \{ \forall f\in F:\|f-f^*\|_{L_2} \leq \max(1,\sqrt{U\|f^*\|_{\cH_K}}) \bar  r(\bar A), \enspace \|f-f^*\|_{\cH_K}^2 \leq 4(2 + 1/\bar A) \|f^*\|_{\cH_K}^2, \\ \big|(P-P_N)\cL_f\big|\leq \frac{\max(1,U\|f^*\|_{\cH_K}) \bar r^2(\bar A)}{2\bar A}   \bigg\}
	\end{align*}
	where we recall that $U = 2L \sqrt{(2+1/\bar A) \|K\|_{\infty}}$. By taking $r^* = \max(1,\sqrt{U\|f^*\|_{\cH_K}}) \bar r(\bar A)$ in the proof of Proposition~\ref{prop:algebra2} it is clear that the deterministic argument is exaclty the same. \\
	Let us show that $\Omega$ holds with probability larger than $1-\exp\big( -(N \bar r^2(\bar A))/(64(\bar AL)^2) \big)$. Let $\cF = \{f \in \cH_K, \enspace \|f-f^*\|_{L_2} \leq \max(1,\sqrt{U\|f^*\|_{\cH_K}}) \bar r(\bar A), \enspace \|f-f^*\|_{\cH_K}^2 \leq \rho^*  \}$. From Assumption~\ref{assum:lip_conv} for all $x,y \in \cX \times \cY$ and $f \in \cF$
	\begin{equation*}
		|(\ell_f - \ell_{f^*})(x,y) | \leq L |f(x)-f^*(x) | \leq  \max(1,U \|f^*\|_{\cH_K})
	\end{equation*}
	We can Therefore use Theorem~\ref{thm:bousquet} with $M = \max(1,U \|f^*\|_{\cH_K})$. From the definition of $\cF$ it follows that $\sup_{f  \in \cF}  P (\ell_f- \ell_{f^*})^2 \leq L^2\max(1,U\|f^*\|_{\cH_K}) \bar r^2(\bar A)$. Let $(\sigma_i)_{i=1}^N$ be i.i.d Rademacher random variables independent from $(X_i,Y_i)_{i=1}$, from the symmetrization and contraction Lemmas \cite{ledoux2013probability} we get
	\begin{align*}
		\bE & \sup_{f  \in \cF} |(P_N-P)\cL_f|  \leq 4L  \bE \sup_{f \in \cF }  \enspace \frac{1}{N} \sum_{i=1}^N \sigma_i (f-f^*)(X_i) \leq  \max(1,U\|f^*\|_{\cH_K})\frac{\bar r^2(\bar A)}{16\bar A}
	\end{align*}
	where we used the Definition~\ref{def:function_r_rkhs} of $r(\cdot)$. For any $t>0$, it follows from Theorem~\ref{thm:bousquet} that for any function $f$ in $\cF$
	\begin{align*}
		|(P_N-P)\cL_f | \leq & \max(1,U\|f^*\|_{\cH_K}) \frac{\bar r^2(\bar A)}{16\bar A} + \frac{\max(1,U\|f^*\|_{\cH_K})t}{3N}\\
		& + \sqrt{\frac{2t}{N}\max(1,U^2\|f^*\|_{\cH_K}^2) \bar  r^2(\bar A) (L^2+\frac{1}{8\bar A} ) }  \enspace.
	\end{align*}
	Take $t = N \bar r^2(\bar A) / (64(L\bar A)^2)$ and use the fact that $\bar A,L \geq 1$ conclude the proof.

	\section{Supplementary lemmas} \label{app_supp}
	
	\begin{Lemma} \label{lem_homo}
		Let $\gamma>0$ and $f$ in $F$ such that $\phi(f-f^*) \geq \gamma$. Then, there exist $f_0$ in $F$ and $1 \leq \alpha \leq  \phi(f-f^*)/ \gamma $ such that $f = f^* + \alpha(f_0 - f^*)$ and $\phi(f_0 - f^*) = \gamma$  
	\end{Lemma}
	\begin{proof}
		Let $\alpha_0 = \sup \{ \alpha > 0, \; \phi \big(\alpha(f-f^*) \big)  \leq \gamma  \}$. For $\alpha = \gamma/ \phi(f-f^*) \leq 1$ we have $\phi \big(\alpha(f-f^*) \big) \leq \alpha \phi (f-f^*)  = \gamma $ so that $\alpha_0 \geq \gamma/\phi(f-f^*)$. By convexity of $F$, $f_0 := f^* + \alpha_0(f-f^*) \in F$ and $\alpha_0 \leq 1$ otherwise, by convexity of $\phi$ we would have $\alpha_0 \phi(f-f^*) \leq \phi \big( \alpha_0 (f-f^*) \big) \leq \gamma$. Moreover, by maximality of $\alpha_0$, $f_0 $ is such that $\phi \big(\alpha(f-f^*) \big) = \phi(f_0-f^*) = \gamma$. The result follows for $\alpha = \alpha^{-1}_0$ 
	\end{proof}	
	\begin{Lemma}\label{lem_conv}
		Let $f : \bR  \mapsto \bR$ be a convex function. Then for all $\lambda \geq 1$ and $x,y$ in $\bR$:
		\begin{equation} \label{conv_reg}
		f(\lambda x + (1-\lambda) y) \geq \lambda f(x) + (1-\lambda) f(y)
		\end{equation}
	\end{Lemma}
	\begin{proof}
		Let $\lambda \geq 1$, by convexity of $f$, for all $x,y$ in $\bR$:
		\begin{equation*}
			f \bigg(\frac{1}{\lambda} x + (1-\frac{1}{\lambda}) y \bigg) \leq \frac{1}{\lambda} f(x) + (1-\frac{1}{\lambda}) f(y)
		\end{equation*}
		It suffice to take $x = \lambda x + (1-\lambda) y$ to get the result.
	\end{proof}

\section*{Acknowledgements}
I would like to thank Guillaume Lecu{\'e} and Matthieu Lerasle for their precious advices on this work. 

\begin{footnotesize}
	\bibliographystyle{amsplain}
	\bibliography{biblio}
\end{footnotesize}

\end{document}

%% file: commandes_guillaume.tex
\usepackage{amsmath}

\usepackage{graphicx}
\usepackage{color}
\usepackage{amsmath}
\usepackage{amssymb}
\usepackage{amscd}
\usepackage{bbm}

\usepackage{tikz}
\newlength\Colsep
\usetikzlibrary{patterns}

\usepackage{hyperref}
\hypersetup{
	bookmarks=true,         
	colorlinks=true,       
	linkcolor=red,          
	citecolor=green,        
	filecolor=magenta,      
	urlcolor=cyan           
}

\usepackage[utf8]{inputenc}
\usepackage{amsthm}
\usepackage{bbm} 
\usepackage[linesnumbered,lined,boxed,commentsnumbered]{algorithm2e}

\usepackage{marginnote}

\newtheorem{Theorem}{Theorem}
\newtheorem{Assumption}{Assumption}
\newtheorem{Definition}{Definition}
\newtheorem{Lemma}{Lemma}
\newtheorem{Proposition}{Proposition}

\newtheorem{Remark}{Remark}

\newcommand{\inr}[1]{\bigl< #1 \bigr>}

\newcommand{\norm}[1]{\left\|#1\right\|}%

\DeclareMathOperator*{\argmin}{argmin}

\def\ds1{\textrm{1\kern-0.25emI}} 

\newcommand \E{\mathbb{E}}
\newcommand \R{\mathbb{R}}

\newcommand \cB{{\cal B}}

\newcommand \cD{{\cal D}}

\newcommand \cF{{\cal F}}

\newcommand \cH{{\cal H}}
\newcommand \cI{{\cal I}}

\newcommand \cL{{\cal L}}

\newcommand \cN{{\cal N}}
\newcommand \cO{{\cal O}}

\newcommand \cR{{\cal R}}
\newcommand \cS{{\cal S}}

\newcommand \cX{{\cal X}}
\newcommand \cY{{\cal Y}}

\newcommand \bE{{\mathbb E}}

\newcommand \bN{{\mathbb N}}

\newcommand \bP{{\mathbb P}}

\newcommand \bR{{\mathbb R}}



%% file: dessin_1.tex
\begin{figure}[h]
\centering
\begin{tikzpicture}[scale=0.5]
\draw (0,0) circle (8cm);
\draw (0,0) circle (7.3cm);
\draw (1.5,0) node {$f^*$};
\filldraw (0,0) circle (0.2cm);
\draw (-10,0) -- (0,10) -- (10,0) -- (0,-10) -- (-10,0);
\draw (12,0) node (A) {$ B_{\rho^{*}}^{\phi} (f^*) $};
\draw (0,7) node (B) {$r^* B_{L_2}$};
\draw (11,6) node (C) {$\big( (\lambda \delta \phi(f^*))/((A^{-1}-\theta)r^*)  \big) B_{L_2}$};
\draw (0,-2) node {$\mathbf{\boldsymbol{B_{\rho^*}^{\phi} (f^*)} \cap (f^*+ (\lambda \delta \phi(f^*)/((A^{-1}-\theta)r^*) B_{L_2})}$};
\draw (-10,8) -- (0,0);
\draw (-11,8) node {$f$};
\filldraw (-10,8) circle (0.2cm);
\draw (-5.5,3.5) node {$f_0$};
\filldraw (-5.6,4.5) circle (0.2cm);

\draw[ultra thick] (-7.7,2.3) -- (-2.3,7.7) arc (110:70:6.8) -- (7.7,2.3) arc (20:-20:6.8) -- (2.3,-7.7) arc (-70:-110:6.8) -- (-7.7, -2.3) arc (-160:-200:6.8);

\end{tikzpicture}
\caption{An homogeneity argument for Lipshitz loss functions: $P_N\cL_f^\lambda>0$ when $P_N\cL_{f_0}^\lambda>0$.}
\label{fig:homogen}
\end{figure}
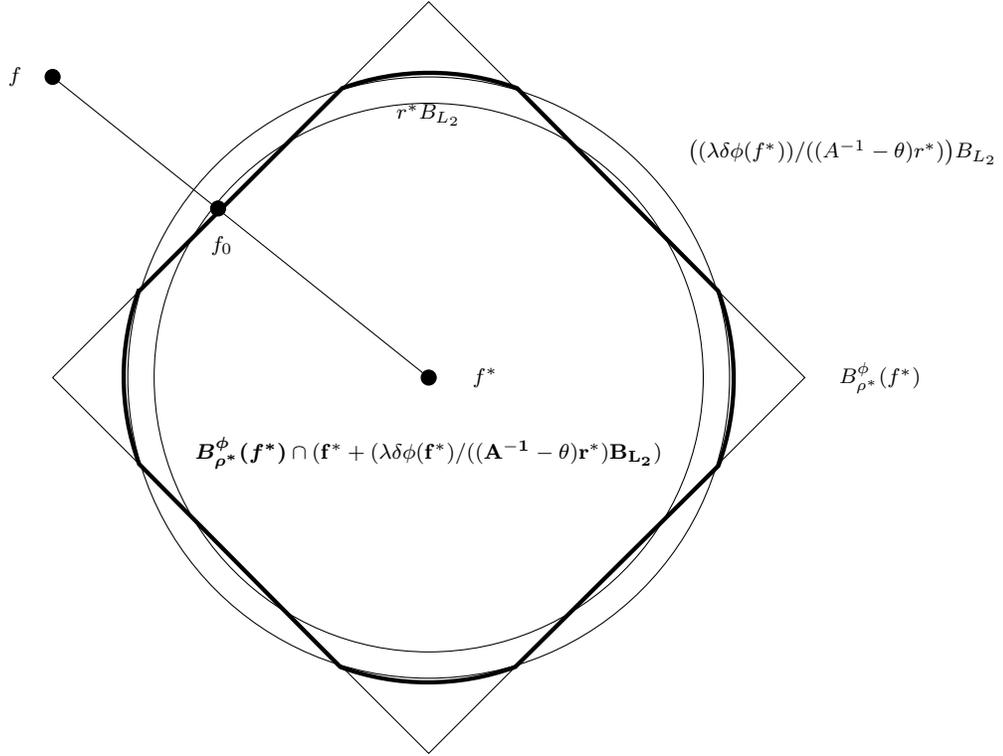

%% file: arxiv.bbl
\providecommand{\bysame}{\leavevmode\hbox to3em{\hrulefill}\thinspace}
\providecommand{\MR}{\relax\ifhmode\unskip\space\fi MR }
\providecommand{\MRhref}[2]{%
  \href{http://www.ams.org/mathscinet-getitem?mr=#1}{#2}
}
\providecommand{\href}[2]{#2}
\begin{thebibliography}{10}

\bibitem{pierre2019estimation}
Pierre Alquier, Vincent Cottet, Guillaume Lecu{\'e}, et~al., \emph{Estimation
  bounds and sharp oracle inequalities of regularized procedures with lipschitz
  loss functions}, The Annals of Statistics \textbf{47} (2019), no.~4,
  2117--2144.

\bibitem{bartlett2006convexity}
Peter~L Bartlett, Michael~I Jordan, and Jon~D McAuliffe, \emph{Convexity,
  classification, and risk bounds}, Journal of the American Statistical
  Association \textbf{101} (2006), no.~473, 138--156.

\bibitem{bellec2017localized}
Pierre~C Bellec et~al., \emph{Localized gaussian width of $ m $-convex hulls
  with applications to lasso and convex aggregation}, Bernoulli \textbf{25}
  (2019), no.~4A, 3016--3040.

\bibitem{bellec2018slope}
Pierre~C Bellec, Guillaume Lecu{\'e}, Alexandre~B Tsybakov, et~al., \emph{Slope
  meets lasso: improved oracle bounds and optimality}, The Annals of Statistics
  \textbf{46} (2018), no.~6B, 3603--3642.

\bibitem{ben2005kernel}
Asa Ben-Hur and William~Stafford Noble, \emph{Kernel methods for predicting
  protein--protein interactions}, Bioinformatics \textbf{21} (2005),
  no.~suppl\_1, i38--i46.

\bibitem{bickel2009simultaneous}
Peter~J Bickel, Ya’acov Ritov, Alexandre~B Tsybakov, et~al.,
  \emph{Simultaneous analysis of lasso and dantzig selector}, The Annals of
  Statistics \textbf{37} (2009), no.~4, 1705--1732.

\bibitem{birge2001alternative}
Lucien Birg{\'e}, \emph{An alternative point of view on lepski's method},
  Lecture Notes-Monograph Series (2001), 113--133.

\bibitem{bishop2006pattern}
Christopher~M Bishop, \emph{Pattern recognition and machine learning},
  springer, 2006.

\bibitem{bousquet2002bennett}
Olivier Bousquet, \emph{A bennett inequality and its application to suprema of
  empirical processes}, Comptes Rendus Mathematique \textbf{334} (2002), no.~6,
  495--500.

\bibitem{caponnetto2007optimal}
Andrea Caponnetto and Ernesto De~Vito, \emph{Optimal rates for the regularized
  least-squares algorithm}, Foundations of Computational Mathematics \textbf{7}
  (2007), no.~3, 331--368.

\bibitem{chafai2012interactions}
Djalil Chafa{\"\i}, Olivier Gu{\'e}don, Guillaume Lecu{\'e}, and Alain Pajor,
  \emph{Interactions between compressed sensing random matrices and high
  dimensional geometry}, Citeseer, 2012.

\bibitem{chalup2008kernel}
Stephan~K Chalup and Andreas Mitschele, \emph{Kernel methods in finance},
  Handbook on information technology in finance, Springer, 2008, pp.~655--687.

\bibitem{chinot2019robust}
Geoffrey Chinot, Guillaume Lecu{\'e}, and Matthieu Lerasle, \emph{Robust high
  dimensional learning for lipschitz and convex losses}, arXiv preprint
  arXiv:1905.04281 (2019).

\bibitem{chinot}
\bysame, \emph{Robust statistical learning with lipschitz and convex loss
  functions}, Probability Theory and Related Fields (2019).

\bibitem{eberts2013optimal}
Mona Eberts, Ingo Steinwart, et~al., \emph{Optimal regression rates for svms
  using gaussian kernels}, Electronic Journal of Statistics \textbf{7} (2013),
  1--42.

\bibitem{farooq2017learning}
Muhammad Farooq and Ingo Steinwart, \emph{Learning rates for kernel-based
  expectile regression}, Machine Learning \textbf{108} (2019), no.~2, 203--227.

\bibitem{golub1979generalized}
Gene~H Golub, Michael Heath, and Grace Wahba, \emph{Generalized
  cross-validation as a method for choosing a good ridge parameter},
  Technometrics \textbf{21} (1979), no.~2, 215--223.

\bibitem{huang2003local}
Jianhua~Z Huang et~al., \emph{Local asymptotics for polynomial spline
  regression}, The Annals of Statistics \textbf{31} (2003), no.~5, 1600--1635.

\bibitem{koltchinskii2011empirical}
Vladimir Koltchinskii, \emph{Empirical and rademacher processes}, Oracle
  Inequalities in Empirical Risk Minimization and Sparse Recovery Problems,
  Springer, 2011, pp.~17--32.

\bibitem{lecue2017robust}
Guillaume Lecu{\'e} and Matthieu Lerasle, \emph{Robust machine learning by
  median-of-means: theory and practice}, arXiv preprint arXiv:1711.10306
  (2017).

\bibitem{lecue2017learning}
\bysame, \emph{Learning from mom’s principles: Le cam’s approach},
  Stochastic Processes and their Applications (2018).

\bibitem{lecue2018robust}
Guillaume Lecu{\'e}, Matthieu Lerasle, and Timoth{\'e}e Mathieu, \emph{Robust
  classification via mom minimization}, arXiv preprint arXiv:1808.03106 (2018).

\bibitem{lecue2017regularization}
Guillaume Lecue and Shahar Mendelson, \emph{Regularization and the small-ball
  method ii: complexity dependent error rates}, The Journal of Machine Learning
  Research \textbf{18} (2017), no.~1, 5356--5403.

\bibitem{lecue2018regularization}
Guillaume Lecu{\'e}, Shahar Mendelson, et~al., \emph{Regularization and the
  small-ball method i: sparse recovery}, The Annals of Statistics \textbf{46}
  (2018), no.~2, 611--641.

\bibitem{ledoux2013probability}
Michel Ledoux and Michel Talagrand, \emph{Probability in banach spaces:
  isoperimetry and processes}, Springer Science \& Business Media, 2013.

\bibitem{lepskii1992asymptotically}
Lepskii, \emph{Asymptotically minimax adaptive estimation. i: Upper bounds.
  optimally adaptive estimates}, Theory of Probability \& Its Applications
  \textbf{36} (1992), no.~4, 682--697.

\bibitem{lepskii1993asymptotically}
\bysame, \emph{Asymptotically minimax adaptive estimation. ii. schemes without
  optimal adaptation: Adaptive estimators}, Theory of Probability \& Its
  Applications \textbf{37} (1993), no.~3, 433--448.

\bibitem{li2007nonparametric}
Qi~Li and Jeffrey~Scott Racine, \emph{Nonparametric econometrics: theory and
  practice}, Princeton University Press, 2007.

\bibitem{LugosiMendelson2016}
Gabor Lugosi and Shahar Mendelson, \emph{Risk minimization by median-of-means
  tournaments}, To appear in JEMS (2016).

\bibitem{LugosiMendelson2017}
G{\'a}bor Lugosi, Shahar Mendelson, et~al., \emph{Regularization, sparse
  recovery, and median-of-means tournaments}, Bernoulli \textbf{25} (2019),
  no.~3, 2075--2106.

\bibitem{lugosi2019sub}
\bysame, \emph{Sub-gaussian estimators of the mean of a random vector}, The
  Annals of Statistics \textbf{47} (2019), no.~2, 783--794.

\bibitem{marsh2001spline}
Lawrence~C Marsh and David~R Cormier, \emph{Spline regression models}, vol.
  137, Sage, 2001.

\bibitem{meister2016optimal}
Mona Meister and Ingo Steinwart, \emph{Optimal learning rates for localized
  svms}, The Journal of Machine Learning Research \textbf{17} (2016), no.~1,
  6722--6765.

\bibitem{mendelson2003performance}
Shahar Mendelson, \emph{On the performance of kernel classes}, Journal of
  Machine Learning Research \textbf{4} (2003), no.~Oct, 759--771.

\bibitem{mendelson2017multiplier}
\bysame, \emph{On multiplier processes under weak moment assumptions},
  Geometric aspects of functional analysis, Springer, 2017, pp.~301--318.

\bibitem{mendelson2010regularization}
Shahar Mendelson, Joseph Neeman, et~al., \emph{Regularization in kernel
  learning}, The Annals of Statistics \textbf{38} (2010), no.~1, 526--565.

\bibitem{minsker2018uniform}
Stanislav Minsker, \emph{Uniform bounds for robust mean estimators}, arXiv
  preprint arXiv:1812.03523 (2018).

\bibitem{noble2004support}
William~Stafford Noble et~al., \emph{Support vector machine applications in
  computational biology}, Kernel methods in computational biology \textbf{71}
  (2004), 92.

\bibitem{scholkopf1999advances}
Bernhard Sch{\"o}lkopf, Christopher~JC Burges, Alexander~J Smola, et~al.,
  \emph{Advances in kernel methods: support vector learning}, MIT press, 1999.

\bibitem{scholkopf2004support}
Bernhard Sch{\"o}lkopf, Koji Tsuda, and Jean-Philippe Vert, \emph{Support
  vector machine applications in computational biology}, MIT press, 2004.

\bibitem{mendelson2014learning}
Mendelson Shahar, \emph{Learning without concentration}, Conference on Learning
  Theory, 2014, pp.~25--39.

\bibitem{shawe2004kernel}
John Shawe-Taylor, Nello Cristianini, et~al., \emph{Kernel methods for pattern
  analysis}, Cambridge university press, 2004.

\bibitem{smale2007learning}
Steve Smale and Ding-Xuan Zhou, \emph{Learning theory estimates via integral
  operators and their approximations}, Constructive approximation \textbf{26}
  (2007), no.~2, 153--172.

\bibitem{steinwart2008support}
Ingo Steinwart and Andreas Christmann, \emph{Support vector machines}, Springer
  Science \& Business Media, 2008.

\bibitem{talagrand2006generic}
Michel Talagrand, \emph{The generic chaining: upper and lower bounds of
  stochastic processes}, Springer Science \& Business Media, 2006.

\bibitem{MR2829871}
Koltchinskii Vladimir, \emph{Oracle inequalities in empirical risk minimization
  and sparse recovery problems}, Lecture Notes in Mathematics, vol. 2033,
  Springer, Heidelberg, 2011, Lectures from the 38th Probability Summer School
  held in Saint-Flour, 2008, \'Ecole d'\'Et\'e de Probabilit\'es de
  Saint-Flour. [Saint-Flour Probability Summer School]. \MR{2829871}

\bibitem{wu2006learning}
Qiang Wu, Yiming Ying, and Ding-Xuan Zhou, \emph{Learning rates of least-square
  regularized regression}, Foundations of Computational Mathematics \textbf{6}
  (2006), no.~2, 171--192.

\bibitem{yang2000face}
M-H Yang, Narendra Ahuja, and David Kriegman, \emph{Face recognition using
  kernel eigenfaces}, Proceedings 2000 International Conference on Image
  Processing (Cat. No. 00CH37101), vol.~1, IEEE, 2000, pp.~37--40.

\bibitem{zhang2004statistical}
Tong Zhang et~al., \emph{Statistical behavior and consistency of classification
  methods based on convex risk minimization}, The Annals of Statistics
  \textbf{32} (2004), no.~1, 56--85.

\bibitem{zou2005regularization}
Hui Zou and Trevor Hastie, \emph{Regularization and variable selection via the
  elastic net}, Journal of the Royal Statistical Society: Series B (Statistical
  Methodology) \textbf{67} (2005), no.~2, 301--320.

\end{thebibliography}
